\documentclass[a4paper,11pt,reqno]{amsart}
\usepackage{amsmath,amssymb,url,stmaryrd}
\usepackage{lmodern}
\usepackage[T1]{fontenc}
\usepackage[top=30truemm,bottom=30truemm,left=20truemm,right=20truemm]{geometry}
\usepackage{tikz}
\usetikzlibrary{positioning}
\usetikzlibrary{arrows.meta}
\usetikzlibrary{calc}
\usepackage{color}
\usepackage{enumitem}
\usepackage{hyperref}

\newenvironment{enumerate2}{\begin{enumerate}[label=\textup{(\alph*)}]}
{\end{enumerate}}

\date{\today}

\newtheorem{Thm}{Theorem}[section]
\newtheorem{Cor}[Thm]{Corollary}
\newtheorem{Lem}[Thm]{Lemma}
\newtheorem{Prop}[Thm]{Proposition}
\newtheorem{Def-Lem}[Thm]{Definition-Lemma}
\newtheorem{Def-Prop}[Thm]{Definition-Proposition}
\newtheorem{Def-Thm}[Thm]{Definition-Theorem}
\newtheorem{Conj}[Thm]{Conjecture}

\theoremstyle{definition}
\newtheorem{Def}[Thm]{Definition}

\newtheorem{Ex}[Thm]{Example}
\newtheorem{Nota}[Thm]{Notation}

\newcommand{\calC}{\mathcal{C}}
\newcommand{\calD}{\mathcal{D}}
\newcommand{\calF}{\mathcal{F}}
\newcommand{\calH}{\mathcal{H}}

\newcommand{\calT}{\mathcal{T}}
\newcommand{\calU}{\mathcal{U}}
\newcommand{\calV}{\mathcal{V}}
\newcommand{\calW}{\mathcal{W}}

\newcommand{\ovcalF}{\overline{\calF}}
\newcommand{\ovcalT}{\overline{\calT}}

\newcommand{\sfD}{\mathsf{D}}
\newcommand{\sfK}{\mathsf{K}}

\newcommand{\rmb}{\mathrm{b}}

\newcommand{\R}{\mathbb{R}}
\newcommand{\Z}{\mathbb{Z}}

\DeclareMathOperator{\Hom}{\mathsf{Hom}}
\DeclareMathOperator{\End}{\mathsf{End}}

\DeclareMathOperator{\Ext}{\mathsf{Ext}}

\DeclareMathOperator{\rad}{\mathsf{rad}}
\DeclareMathOperator{\soc}{\mathsf{soc}}

\DeclareMathOperator{\Ker}{\mathsf{Ker}}
\DeclareMathOperator{\Image}{\mathsf{Im}} \renewcommand{\Im}{\Image}

\DeclareMathOperator{\module}{\mathsf{mod}} \renewcommand{\mod}{\module}

\DeclareMathOperator{\proj}{\mathsf{proj}}

\DeclareMathOperator{\add}{\mathsf{add}}

\DeclareMathOperator{\Filt}{\mathsf{Filt}}
\DeclareMathOperator{\Fac}{\mathsf{Fac}}
\DeclareMathOperator{\Sub}{\mathsf{Sub}}

\DeclareMathOperator{\simple}{\mathsf{sim}}

\DeclareMathOperator{\tors}{\mathsf{tors}}
\DeclareMathOperator{\ftors}{\mathsf{f-tors}}

\DeclareMathOperator{\ftorf}{\mathsf{f-torf}}
\DeclareMathOperator{\brick}{\mathsf{brick}}

\DeclareMathOperator{\smc}{\mathsf{smc}}
\DeclareMathOperator{\twosmc}{2\mathsf{-smc}}
\DeclareMathOperator{\silt}{\mathsf{silt}}
\DeclareMathOperator{\twosilt}{2\mathsf{-silt}}
\DeclareMathOperator{\psilt}{\mathsf{psilt}}
\DeclareMathOperator{\twopsilt}{2\mathsf{-psilt}}

\DeclareMathOperator{\fLsbrick}{\mathsf{f_L-sbrick}}

\DeclareMathOperator{\fRsbrick}{\mathsf{f_R-sbrick}}

\DeclareMathOperator{\sfT}{\mathsf{T}}
\DeclareMathOperator{\sfF}{\mathsf{F}}

\DeclareMathOperator{\SH}{\mathsf{SH}}
\DeclareMathOperator{\reduc}{\mathsf{red}}

\DeclareMathOperator{\Hasse}{\mathsf{Hasse}}

\renewcommand{\epsilon}{\varepsilon}
\renewcommand{\phi}{\varphi}
\renewcommand{\Gamma}{\varGamma}
\renewcommand{\Phi}{\varPhi}
\renewcommand{\Psi}{\varPsi}

\numberwithin{equation}{section}

\begin{document}
\title[Mutations of simple-minded collections revisited]
{Mutations of simple-minded collections revisited}

\author{Sota Asai} 
\address{Sota Asai: Graduate School of Mathematical Sciences,
University of Tokyo,  
3-8-1 Komaba, Meguro-ku, Tokyo-to, 153-8914, Japan}
\email{sotaasai@g.ecc.u-tokyo.ac.jp}

\begin{abstract}
Simple-minded collections in the bounded derived category $\sfD^\rmb(\mod A)$
are necessary tools in tilting theory of finite dimensional algebras,
as well as silting complexes 
in the perfect derived category $\sfK^\rmb(\proj A)$.
In this paper, we give an explicit mutation formula of
simple-minded collections at arbitary direct summands,
which is compatible with that of silting complexes.
Though this mutation formula may be known to experts, 
we provide a thorough proof for completeness and future reference.
Then we use our formula to revisit $\tau$-tilting theory
in terms of mutations of 2-term simple-minded collections,
including their relationship with maximal/minimal completions 
of 2-term presilting complexes and $\tau$-tilting reduction.
\end{abstract}

\maketitle 
\setcounter{tocdepth}{1}
\tableofcontents

\section{Introduction}

In the bounded derived category $\sfD(A)=\sfD^\rmb(\mod A)$
of a finite dimensional algebra $A$ over a field $K$,
\emph{simple-minded collections} introduced by Al-Nofayee \cite{Al-Nofayee}
are a notion generalizing the set of simple modules.
They have been well-studied in many works including recent ones 
\cite{Fushimi,GZ,HW,Jin,Plogmann,SZ}.
By \cite{Al-Nofayee,BBD,IY,KN,KY},
the filtration closure of a simple-minded collection 
is the heart of some bounded \emph{t-structure} in $\sfD(A)$,
and it is equivalent to the module category $\mod E$
for a certain finite dimensional algebra $E$.
Thus simple-minded collections are important
to relate the module categories of many algebras.

This phenomenon is based on the bijections by Koenig-Yang \cite{KY}
between the set $\smc A$ of simple-minded collections in $\sfD(A)$ and 
the set $\silt A$ of \emph{silting complexes} 
in the perfect derived category $\sfK^\rmb(\proj A)$.
Silting complexes introduced by Keller-Vossieck \cite{KV}
contain \emph{tilting complexes}, 
which characterize \emph{derived equivalences} of algebras by \cite{Rickard}.
Unlike tilting complexes, 
silting complexes are closed under \emph{mutations};
that is, we can always change a direct summand 
of a given silting complex in a certain way to get another silting complex.
This was motivated by mutations of seeds in the \emph{cluster theory}
introduced by Fomin-Zelevinsky \cite{FZ}.
Though \emph{irreducible mutations} exchaging 
only one indecomposable direct summand are typical,
Aihara-Iyama \cite[Definition 2.30, Theorem 2.31]{AI} showed that
mutations of a silting complex at an arbitrary direct summand 
are well-defined by the next formula.

\begin{Def-Prop}[Definition-Proposition \ref{Def-Prop_silt_mut}]
Let $S=\bigoplus_{i=1}^n S_i \in \silt A$,
and $J \subset \{1,\ldots,m\}$ be a subset.
Set $S_J:=\bigoplus_{j \in J} S_j$ and $U:=S/S_J$.
\begin{enumerate}
\item
For each $j \in J$, there exists a triangle 
\begin{align*}
S_j \xrightarrow{f_j} U'_j \xrightarrow{f'_j} T_j \to S_j[1]
\end{align*} 
in $\sfK^\rmb(\proj A)$
with $f_j$ a miminal left $(\add U)$-approximation,
$f'_j$ a minimal right $(\add U)$-approximation, and
$T_j$ indecomposable.
\item
We set $T_i:=S_i$ if $i \notin J$, and take $T_j$ for $j \in J$ in (1).
Then the \emph{left mutation} of $S$ at $S_J$
is defined as
\begin{align*}
\mu_J^-(S):=\bigoplus_{i=1}^n T_i \in \silt A,
\end{align*}
including the indices of the indecomposable direct summands.
\end{enumerate}
\end{Def-Prop}

If a silting complex $S \in \silt A$ is sent
to a simple-minded collection $X \in \smc A$
under the bijection $\SH \colon \silt A \to \smc A$,
there is a natural bijection
between the indecomposable direct summands of $S$ and those of $X$.
More precisely, there exist decompositions
$S=\bigoplus_{i=1}^n S_i$ and $X=\bigoplus_{i=1}^n X_i$
into indecomposable direct summands which satisfy the duality
\begin{align*}
\Hom_{\sfD(A)}(S_i,X_{i'}[l]) \simeq
\begin{cases}
\End_{\sfD(A)}(X_i) \quad \text{as left $\End_{\sfD(A)}(X_i)$-modules} 
& ((i',l)=(i,0)) \\
0 & ((i',l) \ne (i,0))
\end{cases};
\end{align*}
see Definition-Proposition \ref{Def-Prop_SH} for detail.
Unless otherwise stated, we assume such decompositions of $S$ and $X$.

Thus there surely exist mutations of simple-minded collections
which are compatible with those of silting complexes,
and \cite{BDH,BCPW,BY,KY} already give interesting observations on them.
However, we could not find a reference for 
the precise statement with enough proof,
which are essential to develop our theory \cite{A-interval}
of interval neighborhoods in the real Grothendieck group $K_0(\proj A)_\R$.
Thus, for the sakeness of completeness and future reference,
we decided to write this paper to give a thorough proof of 
the following properties,
though they may be known to experts.
Section \ref{Sec_general} is devoted to proving them.

\begin{Def-Thm}[Theorems \ref{Thm_smc_mut_approx}, \ref{Thm_smc_mut},
Definition-Proposition \ref{Def-Prop_smc_mut}]
\label{simu mut smc intro}
Let $S \in \silt A$, $X:=\SH(S)=\bigoplus_{i=1}^n X_i \in \smc A$
and $J \subset \{1,\ldots,n\}$.
We set $X_J:=\bigoplus_{j \in J} X_j$ and
$\calW_J:=\Filt \{ X_j \mid j \in J \} \subset \sfD(A)$. 
\begin{enumerate}
\item
If $i \notin J$, then there exists a triangle 
\begin{align*}
X_i[-1]\xrightarrow{g_i} W_i \xrightarrow{g'_i} Y_i \to X_i
\end{align*} 
in $\sfD(A)$ with $g_i$ a minimal left $\calW_J$-approximation,
$g'_i$ a minimal right $\calW_J$-approximation.
Moreover, setting $Y_j:=X_j[1]$ for any $j \in J$,
the \emph{left mutation} $\mu_J^-(X)$ of $X$ at $X_J$ 
is well-defined by
\begin{align*}
\mu_J^-(X):=\bigoplus_{i=1}^n Y_i \in \smc A.
\end{align*}
\item
Set $T:=\mu_J^-(S)$.
Then we have 
\begin{align*}
\SH(T)=\bigoplus_{i=1}^n Y_i=\mu_J^-(X).
\end{align*}
\end{enumerate}
\end{Def-Thm}

In the remaining sections, we focus on 2-term notions.
A simple-minded collection $X=\bigoplus_{i=1}^n X_i \in \smc A$
is said to be \emph{2-term} if $H^l(X)=0$ 
for any $l \in \Z \setminus \{-1,0\}$.
In this case, for each $i \in \{1,\ldots,n\}$,
exactly one of $X_i^+:=H^0(X_i)$ and $X_i^-:=H^{-1}(X_i)$
is a brick, and the other is zero \cite{BY},
cited as Proposition \ref{Prop_stalk} in this paper.
In other words, each $X_i$ belongs to $\mod A$ or $(\mod A)[1]$.

As proved in \cite{AIR,BY,A-semi},
there are many bijections with 2-term simple-minded collections.
The bijection $\SH \colon \silt A \to \smc A$ is restricted to a bijection
$\SH \colon \twosilt A \to \twosmc A$ between the 2-term sliting complexes
and the 2-term simple-minded collections.
Moreover they have one-to-one correspondences with
\emph{functorially finite torsion classes}
and \emph{left finite semibricks} in the module category $\mod A$.
Namely, if $S \in \twosilt A$ and $X=\SH(S) \in \twosmc A$,
then $\Fac H^0(S)$ is a functorially finite torsion class,
and it is the smallest torsion class $\sfT(X^+)$
containing the left finite semibrick $X^+$.
See Propositions \ref{Prop_silt_tors}, 
\ref{Prop_silt_smc_2} and \ref{Prop_sbrick} for detail.

In general,
a left mutation of a 2-term silting complex is not necessarily 2-term.
Thus we give the following characterizations of 
when left mutations can be done within $\twosilt A$,
by using the two functorially finite torsion pairs
\begin{align*}
(\ovcalT_U,\calF_U):=({^\perp H^{-1}(\nu U)},\Sub H^{-1}(\nu U)), \quad
(\calT_U,\ovcalF_U):=(\Fac H^0(U),{H^0(U)^\perp})
\end{align*}
associated to each 2-term presilting complex $U$.
For the case $\#J=1$, this result was obtained in \cite{AIR, A-semi}.

\begin{Thm}[Theorem \ref{Thm_within}]
Let $S=\bigoplus_{i=1}^n S_i \in \twosilt A$, and $J \subset \{1,\ldots,n\}$.
Set $X=\bigoplus_{i=1}^n X_i := \SH(S) \in \twosmc A$,
$T:=\mu_J^-(S)$, $U:=\bigoplus_{i \notin J}S_i$.
Then the following conditions are equivalent.
\begin{enumerate2}
\item
The silting complex $T$ satisfies $T \in \twosilt A$.
\item
The simple-minded collection $\SH(T)$ satisfies $\SH(T) \in \twosmc A$.
\item
For any $j \in J$, we have $X_j^+ \ne 0$; or equivalently, $X_j \in \mod A$.
\item
The torsion pairs $(\ovcalT_S,\calF_S)$ 
coincide with $(\ovcalT_U,\calF_U)$.
\end{enumerate2}
\end{Thm}

As an application of this theorem, 
we prove the following result on maximal semibricks.
We remark that the construction of $X^+$ from $Y$ below
is the same as \cite[Proposition 5.6, Corollary 5.7]{BDH}
for brick finite algebras.

\begin{Prop}[Proposition \ref{Prop_LF_RF_max}]
Let $M \in \fRsbrick A$.
Take the unique $Y=\bigoplus_{i=1}^n Y_i \in \twosmc A$ such that $Y^-=M$,
and set $J:=\{i \in \{1,\ldots,n\} \mid Y_i^- \ne 0\}$.
Then $X:=\mu_J^+(Y) \in \twosmc A$ holds,
and $X^+$ is a left finite, right finite and maximal semibrick
which contains $M$ as a direct summand.
\end{Prop}

Thus, a semibrick $M$ is a direct summand of some left finite semibrick
if and only if $M$ is a direct summand of some right finite semibrick;
see Corollary \ref{Cor_LF_RF_summand}.

In the rest of Section \ref{Sec_2-term},
we recall basic properties of the \emph{exchange quiver} $Q(\twosilt A)$
defined by irreducible mutations of 2-term silting complexes,
as a preparation for the next section.
In \cite{A-semi}, we established the \emph{brick labeling} of $Q(\twosilt A)$ 
from the bijection $\twosilt A \to \twosmc A$ so that
the label of $S \to \mu_j^-(S)$ is $X_j^+ \in \mod A$.
Thus for each $S \in \twosilt A$,
the corresponding $X \in \twosmc A$ is recovered 
by the labels of the arrows involving $S$.

In Section \ref{Sec_reduc}, 
we fix a 2-term presilting complex $U \in \twopsilt A$.
Then there uniquely exist $S,T \in \twosilt A$
such that $\ovcalT_S=\ovcalT_U$ and $\ovcalF_T=\ovcalF_U$.
The 2-term silting complexes $S$ and $T$ are called
the \emph{maximal completion} and the \emph{minimal completion},
respectively.
As in Lemma \ref{Lem_S_T_common}, $T$ is a left mutation of $S$ at $S/U$.
We study their corresponding 
2-term simple-minded collections $X:=\SH(S)$ and $Y:=\SH(T)$ in this section.
For this purpose, we introduce the following canonical notation.

\begin{Nota}[Notation \ref{Nota_S_T}]
Let $U=\bigoplus_{i=1}^m U_i \in \twopsilt A$,
and take its maximal and minimal completions $S$ and $T$ respectively.
We fix decompositions 
$S=\bigoplus_{i=1}^n S_i$ and $T=\bigoplus_{i=1}^n T_i$ 
into indecomposable direct summands satisfying the following properties.
\begin{enumerate2}
\item
For each $i \in \{1,\ldots,m\}$, we have $S_i=U_i=T_i$.
\item
For each $j \in \{m+1,\ldots,n\}$, there exists a triangle
$S_j \xrightarrow{f_j} U'_j \to T_j \to S_j[1]$ in $\sfK^\rmb(\proj A)$ with 
$f_j$ a minimal left $(\add U)$-approximation.
\end{enumerate2}
Set $X=\bigoplus_{i=1}^n X_i:=\SH(S)$, $Y=\bigoplus_{i=1}^n Y_i:=\SH(T)$.
Since $J=\{m+1,\ldots,n\}$, we have $\calW_J=\Filt\{X_{m+1},\ldots,X_n\}$.
Thus the following holds.
\begin{enumerate2}
\item[(c)]
For each $i \in \{1,\ldots,m\}$, we have a triangle
$X_i \xrightarrow{g_i} W_i \to Y_i \to X_i[1]$ in $\sfD(A)$ with
$g_i$ a minimal left $\calW_J$-approximation.
\end{enumerate2}
\end{Nota}

We will show the following properties.

\begin{Thm}[Proposition \ref{Prop_X^-_Y^+_U},
Theorems \ref{Thm_endotop_U} and \ref{Thm_sim_W_U}]
\label{Thm_Sec4_intro}
Let $U=\bigoplus_{i=1}^m U_i \in \twopsilt A$.
\begin{enumerate}
\item
For each $i \in \{1,\ldots,m\}$, we have 
\begin{align*}
Y_i^+=H^0(U_i)/\sum_{f \in \rad_A(H^0(U),H^0(U_i))} \Im f, \quad 
X_i^-=\bigcap_{f \in \rad_A(H^{-1}(\nu U_i),H^{-1}(\nu U))} \Ker f,
\end{align*}
and at least one of them is nonzero.
\item
For each $j \in \{m+1,\ldots,n\}$, we have $X_j^+=Y_j^- \ne 0$.
Moreover these $n-m$ bricks are the nonisomorphic simple objects
of the wide subcategory $\calW_U:=\ovcalT_U \cap \ovcalF_U \subset \mod A$. 
\end{enumerate}
\end{Thm}

We remark that the subcategory $\calW_U$ in (2) is
the \emph{$\tau$-perpendicular category}
in the \emph{$\tau$-tilting reduction} 
(Proposition \ref{Prop_reduc}) introduced by Jasso \cite{Jasso}.
He considered the factor algebra $B=\End_A(H^0(S))/\langle e \rangle$
by the idempotent $H^0(S) \to H^0(U) \to H^0(S)$,
and obtained a category equivalence $\calW_U \to \mod B$
and a canonical bijection $\twosilt_U A \to \twosilt B$,
where $\twosilt_U A$ consists of $V \in \twosilt A$
which has $U$ as a direct summand.

The exchange quiver $Q(\twosilt_U A)$ is defined as just
the full subquiver of $Q(\twosilt A)$
given by the subset $\twosilt_U A \subset \twosilt A$.
We have an isomorphism of $Q(\twosilt_U A) \to Q(\twosilt B)$,
and the brick labeling of these quivers are compatible
with the equivalence $\calW_U \simeq \mod B$ \cite{Jasso,A-semi}.
By Theorem \ref{Thm_Sec4_intro}, 
the arrows around the vertices $S$ and $T$ in $Q(\twosilt A)$
are depicted as follows,
where the gray region denotes $Q(\twosilt_U A)$.
We give detailed explanation 
in and after Proposition \ref{Prop_label_S_T}.
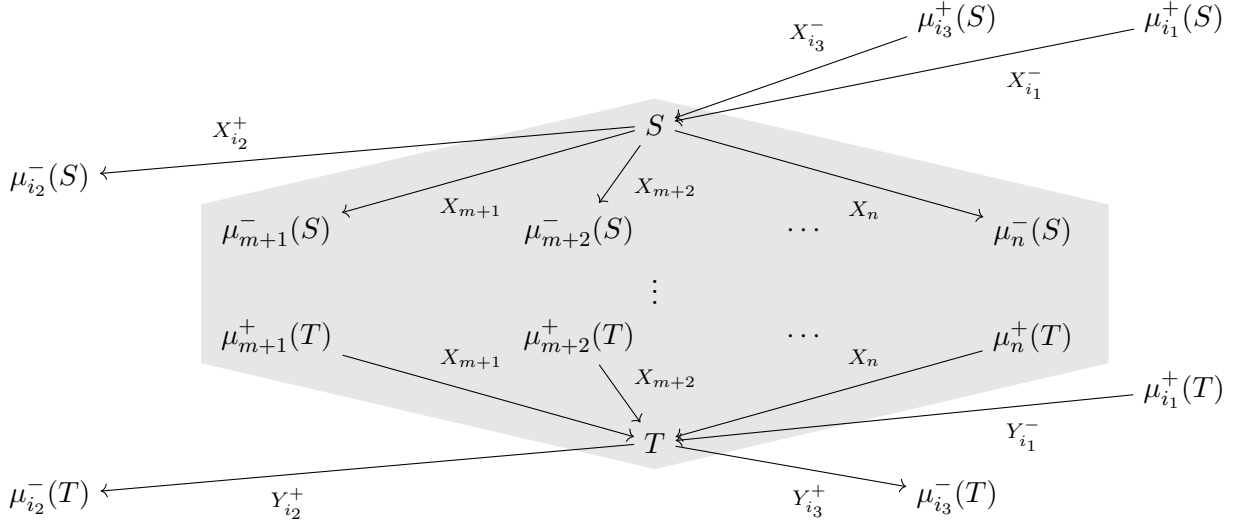
\begin{align*}
\begin{tikzpicture}[->,baseline=0pt,xscale=1, yscale=0.7]
\fill[black!10] (0,3.5)--(-6,1.5)--(-6,-1.5)--
(0,-3.5)--(6,-1.5)--(6,1.5)--cycle;
\node (S0) at ( 0, 3) {$S$};
\node (S1) at ( 7, 5) {$\mu_{i_1}^+(S)$};
\node (S2) at (-8, 2) {$\mu_{i_2}^-(S)$};
\node (S3) at ( 4, 5) {$\mu_{i_3}^+(S)$};
\node (S4) at (-5, 1) {$\mu_{m+1}^-(S)$};
\node (S5) at (-1, 1) {$\mu_{m+2}^-(S)$};
\node (S6) at ( 2, 1) {$\cdots$};
\node (S7) at ( 5, 1) {$\mu_n^-(S)$};
\node (CC) at ( 0, 0) {$\vdots$};
\node (T0) at ( 0,-3) {$T$};
\node (T1) at ( 7,-2) {$\mu_{i_1}^+(T)$};
\node (T2) at (-8,-4) {$\mu_{i_2}^-(T)$};
\node (T3) at ( 4,-4) {$\mu_{i_3}^-(T)$};
\node (T4) at (-5,-1) {$\mu_{m+1}^+(T)$};
\node (T5) at (-1,-1) {$\mu_{m+2}^+(T)$};
\node (T6) at ( 2,-1) {$\cdots$};
\node (T7) at ( 5,-1) {$\mu_n^+(T)$};
\draw (S0) to[edge label ={$\scriptstyle X_{m+1}$},pos=0.7] (S4);
\draw (S0) to[edge label ={$\scriptstyle X_{m+2}$},pos=0.4] (S5);
\draw (S0) to[edge label'={$\scriptstyle X_n$},    pos=0.7] (S7);
\draw (S1) to[edge label ={$\scriptstyle X_{i_1}^-$},pos=0.3] (S0);
\draw (S0) to[edge label'={$\scriptstyle X_{i_2}^+$},pos=0.7] (S2);
\draw (S3) to[edge label'={$\scriptstyle X_{i_3}^-$},pos=0.3] (S0);
\draw (T4) to[edge label ={$\scriptstyle X_{m+1}$},pos=0.3] (T0);
\draw (T5) to[edge label ={$\scriptstyle X_{m+2}$},pos=0.6] (T0);
\draw (T7) to[edge label'={$\scriptstyle X_n$},    pos=0.3] (T0);
\draw (T1) to[edge label ={$\scriptstyle Y_{i_1}^-$},pos=0.3] (T0);
\draw (T0) to[edge label ={$\scriptstyle Y_{i_2}^+$},pos=0.7] (T2);
\draw (T0) to[edge label'={$\scriptstyle Y_{i_3}^+$},pos=0.7] (T3);
\end{tikzpicture}
\end{align*}

\subsection*{Acknowledgement}
The author is grateful to Osamu Iyama for useful discussion.
This work was supported by JSPS KAKENHI Grant Numbers 
JP19K14500, JP20J00088 and JP23K12957. 

\subsection*{Conventions}
In this paper, $A$ is a finite dimensional algebra over a field $K$.

We only consider finitely generated modules.
We write $\mod A$ for the category of finitely generated right $A$-modules,
and its bounded derived category is denoted by $\sfD(A):=\sfD^\rmb(\mod A)$.
Similarly, $\proj A$ means the category of finitely generated projective
right $A$-modules,
and $\sfK^\rmb(\proj A)$ means the perfect derived category.

Let $M$ be an object in $\mod A$, $\sfK^\rmb(\proj A)$ or $\sfD(A)$.
Often we identify the isoclass of each module $M \in \calD$ with $M$ itself.
The symbol $|M|$ denotes the number of isoclasses of
indecomposable direct summands of $M$.
In particular, $|A|$ is the rank of the algebra $A$.
Unless otherwise stated, we set $n:=|A|$ throughout.
We often assume that $M$ is \emph{basic};
that is, $M=\bigoplus_{i=1}^m M_i$ with each $M_i$ indecomposable
and $M_i \not \cong M_j$ for any $i \ne j$.
In this case, we have $|M|=m$.

All subcategories appearing in this paper are assumed to be full
and closed under isomorphisms.
Assume that $\calD$ is $\mod A$, $\sfK^\rmb(\proj A)$ or $\sfD(A)$,
and that $\calC \subset \calD$ is a full subcategory.
Then $\add \calC$ denotes the additive closure of $\calD$;
that is, the subcategory consisting of direct summands of $M^{\oplus s}$ 
with some $M \in \calC$ and $s \in \Z_{\ge 1}$.
Moreover, $\Filt \calC$ is the filtration closure of $\calD$
with respect to short exact sequences if $\calD=\mod A$
and triangles if $\calD=\sfK^\rmb(\proj A),\sfD(A)$.
If $\calC=\{M\}$, we write $\add M:=\add \{M\}$ and so on.

Let $\calC \subset \calD$ be an additive full subcategory, and $M \in \calD$.
Then a \emph{left $\calC$-approximation} of $M$ is a morphism 
$f \colon M \to L$ in $\calD$ such that $L \in \calC$
and that every $f' \colon M \to L'$ with $L' \in \calC$
admits some $g \colon L \to L'$ such that $f'=gf$.
If moreover $f$ does not have $\phi \colon L \to L$ such that 
$\phi$ is not an isomorphism and that $f=\phi f$,
then $f$ is called a \emph{minimal left $\calC$-approximation} of $\calD$.
Dually. \emph{(minimal) right $\calC$-approximations} are defined.

\section{Mutations of simple-minded collections}
\label{Sec_general}

In this section, we aim to show that 
mutations of simple-minded collections in $\sfD(A)$ are compatible with
those of silting complexes in $\sfK^\rmb(\proj A)$.

\subsection{Silting complexes}

We first recall the definition of silting complexes.

\begin{Def}\cite[Subsection 5.1]{KV}
Let $U$ be a complex in $\sfK^\rmb(\proj A)$.
\begin{enumerate}
\item
We say that $U$ is \emph{presilting}
if $\Hom_{\sfK^\rmb(\proj A)}(U,U[l])=0$ for any $l>0$.
\item
We say that $U$ is \emph{silting}
if $U$ is presilting and the thick subcategory of $\sfK^\rmb(\proj A)$ 
generated by $\sfK^\rmb(\proj A)$ is $\sfK^\rmb(\proj A)$ itself.
\end{enumerate}
We write $\psilt A$ (resp.~$\silt A$) for
the set of isoclasses of basic presilting (resp.~silting) complexes
in $\sfK^\rmb(\proj A)$.
\end{Def}

For each $S \in \silt A$,
its indecomposable direct summands give a free basis 
of the Grothendieck group $K_0(\proj A)$,
and hence $|S|=|A|$ by \cite[Theorem 2.27]{AI}.
When we write $S=\bigoplus_{i=1}^n S_i \in \silt A$,
we implicitly assume that each $S_i$ is indecomposable
and $|A|=n$.
A typical example of silting complexes is $A$ 
as a complex concentrated in degree 0,
and its shifts $A[l]$ are also silting complexes.

Clearly, any direct summand of a silting complex is presilting.
However, the converse does not hold in general.
As in \cite{Kalck,Krah},
there exist an algebra $A$ and $U \in \psilt A$ with
$U \notin \silt A$ and $|U|=|A|$,
so the presilting complex $U$ does not have a silting complex $T$
such that $U \in \add T$.
We will deal with 2-term (pre)silting complexes from the next section.
By Proposition \ref{Prop_compl}, any 2-term presilting complex
can be completed to a 2-term silting complex.

Then we recall the mutation formula of silting complexes 
established in \cite{AI}.

\begin{Def-Prop}\label{Def-Prop_silt_mut}
\cite[Definition 2.30, Theorem 2.31]{AI}
Let $S=\bigoplus_{i=1}^n S_i \in \silt A$,
and $J \subset \{1,\ldots,m\}$ be a subset.
Set $S_J:=\bigoplus_{j \in J} S_j$ and $U:=S/S_J$.
\begin{enumerate}
\item
For each $j \in J$, there exists a triangle 
\begin{align*}
S_j \xrightarrow{f_j} U'_j \xrightarrow{f'_j} T_j \to S_j[1]
\end{align*} 
in $\sfK^\rmb(\proj A)$
with $f_j$ a miminal left $(\add U)$-approximation,
$f'_j$ a minimal right $(\add U)$-approximation, and
$T_j$ indecomposable.
\item
\cite[Proposition 2.33]{AI}
We set $T_i:=S_i$ if $i \notin J$, and take $T_j$ for $j \in J$ in (1).
Then the \emph{left mutation} of $S$ at $S_J$
is defined as
\begin{align*}
\mu_J^-(S):=\bigoplus_{i=1}^n T_i,
\end{align*}
including the indices of the indecomposable direct summands.
Dually, we define the \emph{right simultaneous mutation} $\mu_J^+(S)$.
Then we get $\mu_J^+(\mu_J^-(S))=\mu_J^-(\mu_J^+(S))$.
\end{enumerate}
\end{Def-Prop}

If $J=\{j\}$, then we write $\mu_j^-(S):=\mu_J^-(S)$
and $\mu_j^+(S):=\mu_J^+(S)$.
In such cases, we sometimes call $\mu_j^-(S)$ and $\mu_j^+(S)$
\emph{irreducible mutations} to emphasize that $\#J=1$.

We remark that the original statement of \cite{AI} is stated
as mutation of silting subcategories in a more general setting.
To obtain Definition-Proposition \ref{Def-Prop_silt_mut} completely,
we need to show that 
each $T_i$ is indecomposable and that $T=\mu_J^-(S)$ is always basic.
Though these properties are well-known to experts, 
and are stated in \cite{Zhang} for example,
we give a proof below for the convenience of the reader.

\begin{proof}[Proof of Definition-Proposition \ref{Def-Prop_silt_mut}]
By the cited original results of \cite{AI},
we have a triangle $S_J \to U'_J \to T_J \to S_J[1]$
with $S_J \to U'_J$ a minimal left $(\add U)$-approximation,
where $U'_J:=\bigoplus_{j \in J} U'_j$ and $T_J:=\bigoplus_{j \in J}T_j$.

(i)
We show $(\add U) \cap (\add T_J)=\{0\}$
in an analogous way to \cite[Theorem 2.30]{AIR}.

Decompose $T_J$ as $T'' \oplus U''$, where 
$(\add U) \cap  (\add T'')=\{0\}$ and $U'' \in \add U$.
It suffices to show that $T'' \simeq T_J$.
We write $\iota \colon T'' \to T_J$ for the canonical summand embedding.
Then we have a commutative diagram 
\begin{align*}
\begin{tikzpicture}[baseline=0pt,->,scale=1.2]
\node (11) at ( 0, 0.5) {$S_J\vphantom{T''_J}$};
\node (12) at ( 2, 0.5) {$V\vphantom{T''_J}$};
\node (13) at ( 4, 0.5) {$T''\vphantom{T''_J}$};
\node (14) at ( 6, 0.5) {$S_J[1]\vphantom{T''_J}$};
\node (21) at ( 0,-0.5) {$S_J\vphantom{T''_J}$};
\node (22) at ( 2,-0.5) {$U'_J\vphantom{T''_J}$};
\node (23) at ( 4,-0.5) {$T_J\vphantom{T''_J}$};
\node (24) at ( 6,-0.5) {$S_J[1]\vphantom{T''_J}$};
\draw           (11) to [edge label={$\scriptstyle g$}](12);
\draw           (12) to (13);
\draw           (13) to (14);
\draw           (21) to [edge label={$\scriptstyle f$}](22);
\draw           (22) to (23);
\draw           (23) to (24);
\draw[double,-] (11) to (21);
\draw           (12) to [edge label'={$\scriptstyle \gamma$}](22);
\draw           (13) to [edge label'={$\scriptstyle \iota$}] (23);
\draw[double,-] (14)--(24);
\end{tikzpicture}
\end{align*}
with the top and the bottom rows triangles in $\sfK^\rmb(\proj A)$.

By the definition of $\iota$, we have a split triangle
$T'' \xrightarrow{\iota} T_J \to U'' \xrightarrow{0} T''[1]$,
so the commutative diagram and the octahedral axiom give
a triangle $V \xrightarrow{\gamma} U'_J \to U'' \to V[1]$.
Since $U \in (\add S) \cap (\add T)$,
we have $\Hom_{\sfD(A)}(U'',S_J[1])=0$ and $\Hom(U'',T''[1])=0$.
Then the triangle $S_J \xrightarrow{g} V \to T'' \to S_J[1]$ gives 
$\Hom_{\sfD(A)}(U'',V[1])=0$.
Thus the last morphism of 
the triangle $V \xrightarrow{\gamma} U'_J \to U'' \to V[1]$ is zero,
and $\gamma$ is a section.

Thus $g \colon S_J \to V$ satisfies $V \in \add U$ and $f=\gamma g$.
Then since $f$ is a left $(\add U)$-approximation, so is $g$.
By the minimality of $f$, 
the morphism $\gamma \colon V \to U'_J$ is an isomorphism.
Then the commutative diagram above implies that
the summand embedding $\iota \colon T'' \to T_J$ is an isomorphism.

Now $(\add U) \cap (\add T_J)=(\add U) \cap (\add T'')=\{0\}$ follows.

(ii)
Let $j \in J$.
We show that each $T_j$ is indecomposable.
In the triangle $S_j \xrightarrow{f_j} U'_j 
\xrightarrow{f'_j} T_j \to S_j[1]$,
$f'_j$ is a right $(\add U)$-approximation 
by \cite[Proposition 2.33]{AI}.
Since $(\add S_j) \cap (\add U)=\{0\}$,
the morphism $f'_j$ is a minimal right $(\add U)$-approximation.
Moreover $(\add S_j) \cap (\add U)=\{0\}$ implies also $T_j \ne 0$.
Since $(\add T_j) \cap (\add U)=\{0\}$ is proved in (i)
and since $g$ is a minimal right $(\add U)$-approximation,
if $T_j$ is decomposable, then so is $S_j$, a contradiction. 
Thus $T_j$ is indecomposable.

(iii)
We show that $T_J$ is basic.
By (ii), $T_J$ has exactly $\#J$ indecomposable direct summands.
On the other hand, since $U \oplus T_J$ is silting,
we have $|T_J| \ge |A|-|U|=|S_J|=\#J$.
Thus $T_J$ is basic.

By (i) and (iii), $T=U \oplus T_J$ is basic.
\end{proof}

We note the following property.

\begin{Prop}
Under the setting of Definition-Proposition \ref{Def-Prop_silt_mut},
$U=S/S_J$ is the maximal common direct summand of $S$ and $T=\mu_J^-(S)$.
\end{Prop}

\begin{proof}
Clearly, $U$ is a common direct summand of $S$ and $T$.
On the other hand, assume that $V$ is an indecomposable object
such that $V \in (\add S) \cap (\add T)$ and $V \notin \add U$.
Then we can take $j \in J$ with $V \simeq T_j$.
The triangle $S_j \to U'_j \to T_j \to S_j[1]$ gives 
a short exact sequence 
$\Hom_{\sfD(A)}(T_j,U'_j) \to \Hom_{\sfD(A)}(T_j,T_j)
\to \Hom_{\sfD(A)}(T_j,S_j[1])$.
Since $T_j \in \add S$, we have $\Hom_{\sfD(A)}(T_j,S_j[1])=0$.
Thus $\Hom_{\sfD(A)}(T_j,f'_j) \colon
\Hom_{\sfD(A)}(T_j,U'_j) \to \Hom_{\sfD(A)}(T_j,T_j)$ is surjective,
which implies that $T_j$ is a direct summand of $U'_j \in \add U$.
This contradicts the assumption $V \notin \add U$.
Therefore, we get $(\add S) \cap (\add T) \subset \add U$ as desired.
\end{proof}

\subsection{Simple-minded collections}

In this subsection, we deal with mutations of simple-minded collections
without using silting complexes.
We recall the definition of simple-minded collections 
in $\sfD(A)=\sfD^\rmb(\mod A)$.

\begin{Def}\cite[after Lemma 1]{Al-Nofayee}
(cf. \cite[Definition 3.2]{KY} \cite[Theorem 6.1]{Rickard-symm})
Let $X=\bigoplus_{i=1}^m X_i$ be a basic object in $\sfD(A)$
with each $X_i$ indecomposable.
Then we say that $X$ is a \emph{simple-minded collection} in $\sfD(A)$
if the following conditions are satisfied.
\begin{enumerate2}
\item
For each $i \in \{1,\ldots,m\}$, 
the endomorphism ring $\End_{\sfD(A)}(X_i)$ is a division $K$-algebra.
\item
For any $i \ne j \in \{1,\ldots,m\}$,
we have $\Hom_{\sfD(A)}(X_i,X_j)=0$.
\item
For any $i,j \in \{1,\ldots,m\}$ and $l>0$, 
we have $\Hom_{\sfD(A)}(X_i,X_j[-l])=0$.
\item
The thick subcategory of $\sfD(A)$ generated by $X$ is 
$\sfD(A)$ itself.
\end{enumerate2}
The set of isoclasses of simple-minded collections in $\sfD(A)$
is denoted by $\smc A$.
\end{Def}

If $X \in \smc A$, its indecomposable direct summands 
give a free basis of $K_0(\mod A)$,
and hence $|X|=|A|$ by \cite[Corollary 5.1]{KY}.
As in the case of $\silt A$,
by the notation $X=\bigoplus_{i=1}^n X_i$,
we also mean that each $X_i$ is indecomposable
and that $|A|=n$.
A typical example of simple-minded collections is
the direct sum $\bigoplus_{i=1}^n L(i)$ of
all non-isomorphic simple modules $L(1),\ldots,L(n)$.

For each $X=\bigoplus_{i=1}^n X_i \in \smc A$, we define
a pair $(\calU,\calV)$ of full subcategories of $\sfD(A)$by 
\begin{align*}
\calU:=\Filt \{X_i[l] \mid i \in \{1,\ldots,n\}, \ l \ge 0 \}, \quad
\calV:=\Filt \{X_i[-l] \mid i \in \{1,\ldots,n\}, \ l \ge 0 \}.
\end{align*}
By \cite[Proposition 5.4]{KY} and \cite[Proposition 4]{Al-Nofayee},
$(\calU,\calV)$ is a bounded t-structure in $\sfD(A)$
whose heart $\calH:=\calU \cap \calV=\Filt \{X_1,\ldots,X_n\}$
is an abelian length category such that
$X_1,\ldots,X_n$ are its simple objects.
For any subset $J \subset \{1,\ldots,n\}$, we often consider
\begin{align*}
\calW_J:=\Filt\{X_j \mid j \in J\} \subset \sfD(A),
\end{align*}
which is a Serre subcategory of $\calH$.
In particular, also $\calW_J$ is an abelian length category.
For the precise definition and other properties of t-structures,
the reader can refer also to \cite{BBD,KV,BY}.

We will show the following properties,
which gives a sufficient condition for
the existence of mutations of simple-minded collections.

\begin{Thm}\label{Thm_smc_mut_approx}
Let $X=\bigoplus_{i=1}^n X_i \in \smc A$ and $J \subset \{1,\ldots,n\}$.
For each $i \in \{1,\ldots,n\} \setminus J$, assume that
\begin{align*}
X_i[-1]\xrightarrow{g_i} W_i \xrightarrow{g'_i} Y_i \to X_i
\end{align*} 
is a triangle in $\sfD(A)$ with $g_i$ a minimal left $\calW_J$-approximation. 
Moreover set $Y_j:=X_j[1]$ for any $j \in J$.
Then the following statements hold.
\begin{enumerate}
\item
We have $\bigoplus_{i=1}^n Y_i \in \smc A$ with
$\End_{\sfD(A)}(Y_i) \simeq \End_{\sfD(A)}(X_i)$ 
for each $i \in \{1,\ldots,n\}$.
\item
For each $i \in \{1,\ldots,n\} \setminus J$,
the morphism $g'_i \colon W_i \to Y_i$ is 
a minimal right $\calW_J$-approximation.
\end{enumerate}
\end{Thm}

We remark that (the dual of) this theorem already appeared 
in \cite[Theorem 2.7 (iv)$\Rightarrow$(v)]{BCPW} without a proof,
and proved for $\#J=1$ in \cite[Proposition 7.6 (b)]{KY}
and for 2-term simple-minded collections in \cite[Theorem 4.6]{BDH}.
However, we think that it is very important to give a thorough proof
in the general setting, 
which is one of the aims of this paper.

Before showing Theorem \ref{Thm_smc_mut_approx},
we state its direct consequence.

\begin{Def-Prop}\label{Def-Prop_smc_mut}
Let $X=\bigoplus_{i=1}^n X_i \in \smc A$ and $J \subset \{1,\ldots,n\}$.
Set $X_J:=\bigoplus_{j \in J}X_j$.
\begin{enumerate}
\item
Assume that each $i \in \{1,\ldots,n\} \setminus J$ admits a triangle 
$X_i[-1] \to W_i \to Y_i \to X_i$
with the first map a minimal left $\calW_J$-approximation.
Setting $Y_j:=X_j[1]$ for each $j \in J$,
the \emph{left mutation} $\mu_J^-(X)$ of $X$ at $X_J$ is defined by
\begin{align*}
\mu_J^-(X):=\bigoplus_{i=1}^n Y_i \in \smc A,
\end{align*}
including the indices of indecomposable direct summands.
\item
Dually, assume that each $i \in \{1,\ldots,n\} \setminus J$ admits a triangle 
$X_i \to Z_i \to W_i \to X_i[1]$
with the last map a minimal right $\calW_J$-approximation.
Setting $Z_j:=X_j[-1]$ for each $j \in J$,
the \emph{right mutation} $\mu_J^+(X)$ of $X$ at $X_J$ is defined by
\begin{align*}
\mu_J^+(X):=\bigoplus_{i=1}^n Z_i \in \smc A,
\end{align*}
including the indices of indecomposable direct summands.
\item
If $\mu_J^-(X)$ is well-defined by (1), 
then so is $\mu_J^+(\mu_J^-(X))$ by (2), 
and we have $\mu_J^+(\mu_J^-(X))=X$.
Dually, if $\mu_J^+(X)$ is well-defined by (2),
then so is $\mu_J^-(\mu_J^+(X))$ by (1),
and we have $\mu_J^-(\mu_J^+(X))=X$. 
\end{enumerate}
\end{Def-Prop}

Now we start to prove Theorem \ref{Thm_smc_mut_approx}.
We begin with the following triangulated version of Wakamatsu's Lemma.

\begin{Lem}\label{Lem_Wakamatsu}\cite[Lemma 2.1]{Jorgensen}
Assume that $\calW \subset \sfD(A)$ is a subcategory closed under extensions.
Let $X[-1]\xrightarrow{g} W \to Y \to X$ be a triangle in $\sfD(A)$
with $g$ a minimal $\calW$-approximation.
Then $\Hom_{\sfD(A)}(Y,\calW[1])=0$ holds.
\end{Lem}

\begin{proof}
For the convenience of the reader,
we rewrite the proof of \cite[Lemma 2.1]{Jorgensen} for our case.

Let $W' \in \calW$ and $h \colon Y \to W'[1]$.
We have a commutative diagram 
\begin{align*}
\begin{tikzpicture}[baseline=0pt,->,scale=1.2]
\node (11) at ( 0, 0.5) {$W\vphantom{T'[}$};
\node (12) at ( 2, 0.5) {$Y\vphantom{T'[}$};
\node (13) at ( 4, 0.5) {$X\vphantom{T'[}$};
\node (14) at ( 6, 0.5) {$W[1]\vphantom{T'[}$};
\node (21) at ( 0,-0.5) {$W\vphantom{T'[}$};
\node (22) at ( 2,-0.5) {$W'[1]\vphantom{T'[}$};
\node (23) at ( 4,-0.5) {$W''[1]\vphantom{T'[}$};
\node (24) at ( 6,-0.5) {$W[1]\vphantom{T'[}$};
\draw           (11) to [edge label={$\scriptstyle g'$}](12);
\draw           (12) to [edge label={$\scriptstyle g''$}](13);
\draw           (13) to [edge label={$\scriptstyle g[1]$}](14);
\draw           (21) to [edge label={$\scriptstyle hg'$}](22);
\draw           (22) to [edge label={$\scriptstyle s$}](23);
\draw           (23) to [edge label={$\scriptstyle r$}](24);
\draw[double,-] (11) to (21);
\draw           (12) to [edge label'={$\scriptstyle h$}](22);
\draw           (13) to [edge label'={$\scriptstyle h'$}] (23);
\draw[double,-] (14)--(24);
\end{tikzpicture}
\end{align*}
with the top and the bottom rows triangles in $\sfD(A)$.
Then $W',W \in \calW$ implies $W'' \in \calW$.

Since $g[1] \colon X \xrightarrow{h'} W''[1]
\xrightarrow{r} W[1]$ is a minimal left $\calW[1]$-approximation
with $W''[1] \in \calW[1]$, 
the morphism $h'$ is a left $\calW[1]$-approximation
and $r$ is a retraction.
Thus $s$ is a section.

On the other hand, since $g[1]$ is a left $\calW[1]$-approximation,
we take $t \colon W[1] \to W''[1]$ such that $h'=tg[1]$.
Then $sh=h'g''=tg[1]g''=0$.

These imply $h=0$ as desired.
\end{proof}

Then we have the following result.
Note that (a) is the first assumption and 
that (b) implies the second assumption of \cite[Proposition 7.6 (b)]{KY}
in the case $\#J=1$.

\begin{Lem}\label{Lem_ortho_0_1}
Let $X=\bigoplus_{i=1}^n X_i \in \smc A$,
$J \subset \{1,\ldots,n\}$, and
$\calW_J:=\Filt \{ X_j \mid j \in J \}\subset\sfD(A)$.
Assume $i \in \{1,\ldots,n\} \setminus J$, and that
\begin{align}\label{eq_tri_smc}
X_i[-1]\xrightarrow{g_i} W_i \xrightarrow{g'_i} Y_i \to X_i
\end{align} 
is a triangle in $\sfD(A)$ with $W_i \in \calW_J$.
The following conditions are equivalent.
\begin{enumerate2}
\item
The morphism $g_i$ is a minimal left $\calW_J$-approximation.
\item
The functorial morphism 
\begin{align*}
\Hom_{\sfD(A)}(g_i,?) \colon 
\Hom_{\sfD(A)}(W_i,?)\to\Hom_{\sfD(A)}(X_i[-1],?).
\end{align*}
is isomorphic on $\calW_J$, and injective on $\calW_J[1]$.
\item
The conditions
$\Hom_{\sfD(A)}(Y_i,\calW_J)=0$ and $\Hom_{\sfD(A)}(Y_i,\calW_J[1])=0$ hold.
\end{enumerate2}
In this case, $g'_i$ is a minimal right $\calW_J$-approximation.
\end{Lem}

\begin{proof}
(a)$\Rightarrow$(b)
On $\calW_J$, the functorial morphism $\Hom_{\sfD(A)}(g_i,?)$ 
is surjective by (a).

We next show the injectivity of $\Hom_{\sfD(A)}(g_i,?)$ on $\calW_J$.
Let $W \in \calW_J$ and $h \colon W_i \to W$ satisfy $hg_i=0$.
In the abelian category $\calW_J$, we consider the kernel $\Ker h$ 
and the canonical embedding $\iota \colon \Ker h \to W_i$.
Then $hg_i=0 \colon X_i[-1] \to W_i$ implies that 
there exists some $t \colon X_i[-1] \to \Ker h$ such that $g_i=\iota t$.
By $\Ker h \in \calW_J$, 
the minimality of (a) implies that $\iota$ is a retraction.
Then $\iota$ is an isomorphism in $\calW_J$,
because $\iota$ is a monomorphism. 
Therefore $\Ker h=W_i$ and $h=0$ as desired.

Thus $\Hom_{\sfD(A)}(g_i,?)$ is isomorphic on $\calW_J$.

It remains to prove that $
\Hom_{\sfD(A)}(g_i,?)$ is injective on $\calW_J[1]$.
Let $W \in \calW_J$.
The triangle \eqref{eq_tri_smc} gives an exact sequence
\begin{align*}
&\Hom_{\sfD(A)}(Y_i,W[1]) \to 
\Hom_{\sfD(A)}(W_i,W[1]) \xrightarrow{\Hom_{\sfD(A)}(g_i,W[1])} 
\Hom_{\sfD(A)}(X_i[-1],W[1]).
\end{align*}
By (a) and Lemma \ref{Lem_Wakamatsu}, we have $\Hom_{\sfD(A)}(Y_i,W[1])=0$.
Thus $\Hom_{\sfD(A)}(g_i,?)$ is injective on $\calW_J[1]$.

(b)$\Rightarrow$(c)
Let $W \in \calW_J$.
The triangle \eqref{eq_tri_smc} gives an exact sequence
\begin{align*}
\Hom_{\sfD(A)}(X_i,W) \to
\Hom_{\sfD(A)}(Y_i,W) \to \Hom_{\sfD(A)}(W_i,W) \to \Hom_{\sfD(A)}(X_i[-1],W).
\end{align*}
Since $W \in \calW_J$, (b) implies that the last map is injective.
Moreover $W \in \calW_J=\Filt\{X_j \mid j \in J\}$ and $i \notin J$
give $\Hom_{\sfD(A)}(X_i,W)=0$. 
Thus we have $\Hom_{\sfD(A)}(Y_i,W)=0$.

The triangle \eqref{eq_tri_smc}
gives also an exact sequence
\begin{align*}
&\Hom_{\sfD(A)}(W_i[1],W[1]) \to \Hom_{\sfD(A)}(X_i,W[1]) \to
\Hom_{\sfD(A)}(Y_i,W[1]) \to \\
&\Hom_{\sfD(A)}(W_i,W[1]) \to \Hom_{\sfD(A)}(X_i[-1],W[1]).
\end{align*}
(b) implies that the first map is surjective, 
and that the last map is injective.
Therefore we have $\Hom_{\sfD(A)}(Y_i,W[1])=0$.

(c)$\Rightarrow$(a)
Since $\Hom_{\sfD(A)}(Y_i,\calW_J[1])=0$ by (c), 
we have $g_i$ is a left $\calW_J$-approximation by \eqref{eq_tri_smc}. 
If $g_i$ is not minimal, then $W_i$ and $Y_i$
in the triangle \eqref{eq_tri_smc} have a nonzero common direct summand,
but this contradicts $\Hom_{\sfD(A)}(Y_i,\calW_J)=0$ following from (c).

Now the equivalence of (a)--(c) is proved.
The last statement follows from 
$\Hom_{\sfD(A)}(\calW_J,X_i)=0$, $\Hom_{\sfD(A)}(\calW_J,X_i[-1])=0$
and the dual argument of (c)$\Rightarrow$(a).
\end{proof}

Now we can prove Theorem \ref{Thm_smc_mut_approx}.

\begin{proof}[Proof of Theorem \ref{Thm_smc_mut_approx}]
(1)
Recall $\calH=\Filt \{X_1,\ldots,X_n\}$.
We have $\Hom_{\sfD(A)}(\calH,\calH[<0])=0$.
For each $i \notin J$, we write the given triangle
\begin{align}\label{eq_tri_smc_2}
X_i[-1]\xrightarrow{g_i} W_i \xrightarrow{g'_i} Y_i \xrightarrow{g''_i} X_i.
\end{align}

(i)
We show $\End_{\sfD(A)}(Y_i) \simeq \End_{\sfD(A)}(X_i)$ 
for each $i \in \{1,\ldots,n\}$.
This follows from (i-a) and (i-b) below.

(i-a)
Let $i \notin J$.
For any $h \in \End_{\sfD(A)}(Y_i)$, 
there exists some $h' \in \End_{\sfD(A)}(X_i)$ such that
$h'g''_i=g''_ih$, because $\Hom_{\sfD(A)}(W_i,X_i)=0$ 
follows from $W_i \in \calW_J$.
Such $h'$ is unique, 
because $\End_{\sfD(A)}(X_i)$ is a division $K$-algebra
and $g''_i \ne 0$.
Thus we have a canonical homomorphism 
$\End_{\sfD(A)}(Y_i) \to \End_{\sfD(A)}(X_i)$.

We show the injectivity.
If $h \in \End_{\sfD(A)}(Y_i)$ satisfies $g''_ih=0$,
then the triangle \eqref{eq_tri_smc_2} yields that
there exists some $t \colon Y_i \to W_i$ such that $h=g'_it$,
but $t=0$ by Lemma \ref{Lem_ortho_0_1}; hence $h=0$.

Next we check the surjectivity.
Let $h' \in \End_{\sfD(A)}(X_i)$.
Then $g_i[1]h'g''_i \colon Y_i \to W_i[1]$ is zero 
by Lemma \ref{Lem_ortho_0_1}.
Thus the triangle \eqref{eq_tri_smc_2} implies that
we can take $h \in \End_{\sfD(A)}(Y_i)$ such that $h'g''_i=g''_ih$.

Therefore the homomorphism $\End_{\sfD(A)}(Y_i) \to \End_{\sfD(A)}(X_i)$
is an isomorphism.

(i-b)
If $j \in J$, then $\End_{\sfD(A)}(Y_j) \simeq \End_{\sfD(A)}(X_j)$ 
is trivial.

(ii)
We show $\Hom_{\sfD(A)}(Y_i,Y_{i'})=0$ for each $i \ne i' \in \{1,\ldots,n\}$.
This follows from (ii-a)--(ii-d) below.

(ii-a)
Let $i,i' \notin J$. 
Then $i \ne i'$ gives $\Hom_{\sfD(A)}(X_i,X_{i'})=0$, 
and $W_i \in \calW_J$ yields $\Hom_{\sfD(A)}(W_i,X_{i'})=0$.
Thus the triangle \eqref{eq_tri_smc_2} gives 
$\Hom_{\sfD(A)}(Y_i,X_{i'})=0$.
By Lemma \ref{Lem_ortho_0_1},
we get $\Hom_{\sfD(A)}(Y_i,W_{i'})=0$.
These and \eqref{eq_tri_smc_2} yield $\Hom_{\sfD(A)}(Y_i,Y_{i'})=0$.

(ii-b)
Let $i \notin J$ and $i'=:j' \in J$.
Then we get $\Hom_{\sfD(A)}(Y_i,Y_{i'})=\Hom_{\sfD(A)}(Y_i,X_{j'}[1])$.
This is zero by $X_{j'} \in \calW_J$ and Lemma \ref{Lem_ortho_0_1}.

(ii-c)
If $i=:j \in J$ and $i' \notin J$, then we get
$\Hom_{\sfD(A)}(Y_i,Y_{i'})=\Hom_{\sfD(A)}(X_j[1],Y_{i'})$.
This is zero, since $Y_{i'} \in \calH$ by \eqref{eq_tri_smc_2} 
and since $X_j[1] \in \calH[1]$ is clear.

(ii-d)
If $i=:j \in J$ and $i'=:j' \in J$, then we have 
$\Hom_{\sfD(A)}(Y_i,Y_{i'})=\Hom_{\sfD(A)}(X_j[1],X_{j'}[1])=0$,
because $j \ne j'$.

(iii)
We show $\Hom_{\sfD(A)}(Y_i,Y_{i'}[-l])=0$ 
for each $i,i' \in \{1,\ldots,n\}$ and $l \in \Z_{>0}$.
This follows from (iii-a)--(iii-d).

(iii-a)
Let $i,i' \notin J$.
Then $Y_i,Y_{i'} \in \calH$ by \eqref{eq_tri_smc_2}, 
so $\Hom_{\sfD(A)}(Y_i,Y_{i'}[-l])=0$.

(iii-b)
Let $i \notin J$ and $i'=:j' \in J$.
Then $Y_{i'}=X_{j'}[1]$ holds, so we have
$\Hom_{\sfD(A)}(Y_i,Y_{i'}[-l])=\Hom_{\sfD(A)}(Y_i,X_{j'}[-l+1])$.
This is zero by $X_{j'} \in \calW_J$ and Lemma \ref{Lem_ortho_0_1} if $l=1$;
and by $Y_i \in \calH$ and $X_{j'}[-l] \in \calH[<0]$ if $l\ne 1$.

(iii-c)
Let $i=:j \in J$ and $i' \notin J$.
Then $Y_i=X_j[1] \in \calH[1]$ holds, so we have
$\Hom_{\sfD(A)}(Y_i,Y_{i'}[-l])=\Hom_{\sfD(A)}(X_j[1],Y_{i'}[-l])$.
This is zero, 
since \eqref{eq_tri_smc_2} implies $Y_{i'}[-l] \in \calH[<0]$
and since $X_j[1] \in \calH[1]$.

(iii-d)
Let $i=:j \in J$ and $i'=:j' \in J$.
Then we have
$\Hom_{\sfD(A)}(Y_i,Y_{i'}[-l])=\Hom_{\sfD(A)}(X_j[1],X_{j'}[-l+1])=0$.

(iv)
Set $\calC$ as the smallest thick subcategory of $\sfD(A)$
containing $\{Y_1,\ldots,Y_n\}$. 
Then $X_j=Y_j[-1] \in \calC$ for any $j \in J$,
so $\calW_J \subset \calC$ holds.
For any $i \notin J$, we have
$W_i \in \calW_J \subset \calC$ and $Y_i \in \calC$, 
so \eqref{eq_tri_smc_2} gives $X_i \in \calC$.
Thus we obtain $X_1,\ldots,X_n \in \calC$.
Then $X \in \smc A$ implies $\calC=\sfD(A)$.

By (i)--(iv), $Y=\bigoplus_{i=1}^n Y_n \in \smc A$ holds.

(2)
is immediate by the last statement of Lemma \ref{Lem_ortho_0_1}.
\end{proof}

\subsection{The bijection}

Next we recall the bijection by \cite{KY} 
between $\twosilt A$ and $\twosmc A$.
For each $S \in \silt A$, we set
\begin{align*}
\calH_S:=\{ M \in \sfD(A) \mid 
\text{for all $l \ne 0$, $\Hom_{\sfD(A)}(S,M[l])=0$} \}.
\end{align*}
By \cite[Lemmas 3.3, 5.3]{KY},
$\calH_S$ is an abelian length category
which is the heart of the bounded t-structure
$(\calU_S,\calV_S)$ in $\sfD(A)$, where
\begin{align*}
\calU_S&:=\{ M \in \sfD(A) \mid 
\text{$\Hom_{\sfD(A)}(S,M[l])=0$ for all $l>0$} \}, \\
\calV_S&:=\{ M \in \sfD(A) \mid 
\text{$\Hom_{\sfD(A)}(S,M[-l])=0$ for all $l>0$} \}.
\end{align*}
We set $\simple \calH_S$ as the set of isoclasses of its simple objects.
Then the bijection is given as follows.

\begin{Prop}\label{Prop_silt_smc}
\cite[Theorem 6.1]{KY} 
There exists a bijection $\silt A \to \smc A$ given by 
\begin{align*}
S \mapsto \bigoplus_{M \in \simple \calH_S}M.
\end{align*}
\end{Prop}

For each $S=\bigoplus_{i=1}^n S_i \in \silt A$ and 
the corresponding $X=\bigoplus_{i=1}^n X_i \in \smc A$ above,
we fix indices of their indecomposable direct summands as follows.

\begin{Def-Prop}\label{Def-Prop_SH}
\cite[Lemma 5.3]{KY}
Let $S=\bigoplus_{i=1}^n S_i \in \silt A$,
and take $X \in \smc A$ corresponding to $S$ 
in Proposition \ref{Prop_silt_smc}.
Then we have a unique decomposition $X=\bigoplus_{i=1}^n X_i$ so that
\begin{align*}
\Hom_{\sfD(A)}(S_i,X_{i'}[l]) \simeq
\begin{cases}
\End_{\sfD(A)}(X_i) \quad \text{as left $\End_{\sfD(A)}(X_i)$-modules} 
& ((i',l)=(i,0)) \\
0 & ((i',l) \ne (i,0))
\end{cases}.
\end{align*}
Including the indices of the indecomposable direct summands, we set
\begin{align*}
\SH(S):=X=\bigoplus_{i=1}^n X_i.
\end{align*} 
\end{Def-Prop}

We give a fundamental example.

\begin{Ex}\label{Ex_SH(A)}
Assume that $A$ is a basic algebra
with $P(1),\ldots,P(n)$ all non-isomorphic indecomposable projective modules.
Then $A=\bigoplus_{i=1}^n P(i) \in \silt A$ holds,
and it is sent to $\SH(A)=\bigoplus_{i=1}^n L(i) \in \smc A$,
where $L(i)$ is the simple top of $P(i)$.
\end{Ex}

\subsection{The compability of mutations}

The rest of this section is devoted to proving the following properties;
namely, left mutations of simple-minded collections are always possible,
and compatible with those of silting complexes.

\begin{Thm}\label{Thm_smc_mut}
Let $S \in \silt A$, $X:=\SH(S)=\bigoplus_{i=1}^n X_i \in \smc A$
and $J \subset \{1,\ldots,n\}$.
We set 
$\calW_J:=\Filt \{ X_j \mid j \in J \} \subset \calH_S \subset \sfD(A)$. 
\begin{enumerate}
\item
If $i \notin J$, then there exists a triangle 
\begin{align*}
X_i[-1]\xrightarrow{g_i} W_i \xrightarrow{g'_i} Y_i \to X_i
\end{align*} 
in $\sfD(A)$ with $g_i$ a minimal left $\calW_J$-approximation,
$g'_i$ a minimal right $\calW_J$-approximation.
Thus setting $Y_j:=X_j[1]$ for any $j \in J$,
Definition-Proposition \ref{Def-Prop_smc_mut} implies that
$\mu_J^-(X)=\bigoplus_{i=1}^n Y_i \in \smc A$ is well-defined.
\item
Set $T:=\mu_J^-(S)$.
Then we have 
\begin{align*}
\SH(T)=\bigoplus_{i=1}^n Y_i=\mu_J^-(X) \in \smc A.
\end{align*}
\end{enumerate}
\end{Thm}

We prove Theorem \ref{Thm_smc_mut} in the rest of this section.
The case $\#J=1$ was already shown in \cite[Theorem 7.12]{KY}.
In view of Theorem \ref{Thm_smc_mut_approx},
the remaining tasks are
showing the existence of a minimal left $\calW_J$-approximation 
$g_i \colon X_i[-1] \to W_i$
and its compatibility with the corresponding mutation of silting complexes. 
Our strategies roughly follow the proofs of
\cite[Proposition 7.6, Remark 7.7, Lemma 7.8]{KY},
so we begin with the following property.

\begin{Lem}\label{Lem_M_i}
Assume that $A=\bigoplus_{i=1}^n P(i)$ is basic.
Let $J \subset \{1,\ldots,n\}$, $Q:=A/{\bigoplus_{j \in J}P(j)}$
and $\calW_J:=\Filt\{L(j) \mid j \in J\}$.
For each $i \in \{1,\ldots,n\} \setminus J$, the module
\begin{align*}
M_i:=P(i)/\sum_{f \in \rad_A(Q,P(i))} \Im f
\end{align*}
satisfies the following statements.
\begin{enumerate}
\item
We have $\Hom_A(M_i,\calW_J)=0$, $\Ext_A^1(M_i,\calW_J)=0$
and a short exact sequence 
$0 \to M'_i \to M_i \to L(i) \to 0$ with $M'_i \in \calW_J$.
\item
Let $j \in J$, and take a minimal left $(\add Q)$-approximation 
$f_j \colon P(j) \to Q'_j$ of $P(j)$.
Then there exists an isomorphism 
\begin{align*}
\Hom_A(f_j,M_i) \colon \Hom_A(Q'_j,M_i) \to \Hom_A(P(j),M_i).
\end{align*}
\end{enumerate}
\end{Lem}

\begin{proof}
We take the short exact sequence
\begin{align*}
0 \to N \xrightarrow{\iota} P(i) \xrightarrow{\pi} M_i \to 0 \quad
\left( N:= \sum_{\psi \in \rad_A(Q,P(i))} \Im \psi \right).
\end{align*}

(1)
The part $\Hom_A(M_i,\calW_J)=0$ is clear, 
since $M_i$ is a factor module of $P(i)$.

We check $\Ext_A^1(M_i,\calW_J)=0$.
For each $j \in J$,
the short exact sequence above gives a short exact sequence
$\Hom_A(N,L(j)) \to \Ext_A^1(M_i,L(j)) \to \Ext_A^1(P(i),L(j))$.
The first term is zero, 
since $N$ is a factor module of $Q^{\oplus s}$ for some $s$.
The last term is clearly zero.
Thus $\Ext_A^1(M_i,L(j))=0$ holds.

Since $P(i) \notin \add Q$, we have $N \not \subset P(i)$.
Thus $M_i$ is a nonzero factor module of $P(i)$,
so we have a short exact sequence 
$0 \to M'_i \xrightarrow{\phi} M_i \to L(i) \to 0$
with $\phi$ in the radical.
To show $M'_i \in \calW_J$, it is enough to check $\Hom_A(Q,M'_i)=0$.

Let $h \colon Q \to M'_i$.
Since $\phi$ is in the radical, so is $\phi h \colon Q \to M_i$.
On the other hand, there exists $\widetilde{h} \colon Q \to P(i)$
such that $\phi h=\pi \widetilde{h}$.
Since $\phi h$ is not surjective and $\pi$ is surjective,
$\widetilde{h}$ is not surjective.
Thus $\widetilde{h}$ is in the radical, then $\pi \widetilde{h}=0$ holds.
Since $\phi$ is injective and $\alpha h=\pi \widetilde{h}=0$,
we have $h=0$.

(2)
We first show the surjectivity.
Let $h \colon P(j) \to M_i$.
Since $P(j)$ is projective and $\pi$ is surjective, there exists some 
$\widetilde{h} \colon P(j) \to P(i)$ such that $h=\pi \widetilde{h}$.
Since $f_j$ is a left $(\add Q)$-approximation,
$\widetilde{h}$ admits $t \colon Q'_j \to P(i)$ 
such that $\widetilde{h}=tf_j$.
Then we get $h=\pi t f_j$ as desired.

For the injectivity, assume the contrary that there exists
$h \colon Q'_j \to M_i$ such that $h \ne 0$ and $hf_j=0$.
Since $Q'_j$ is projective and $\pi$ is surjective, there exists some 
$r \colon Q'_j \to P(i)$ such that $h=\pi r$.

Since $h \ne 0$, we have $\Im r \not \subset N$,
so $r$ cannot be in the radical.
Then since $P(i)$ is indecomposable, $r$ is a retraction.
Thus we take a section $s \colon P(i) \to Q'_j$ such that
$rs$ is $1_{P(i)}$.
Then $Q'_j=\Im s \oplus (Q'_j/{\Im s})$ holds.

On the other hand, since $\pi rf_j=hf_j=0$,
there exists $t \colon P(j) \to N$ such that $\iota t=rf_j$.
Since $N$ is a factor module of $Q^{\oplus s}$ for some $s$,
and since $f_j$ is a left $(\add Q)$-approximation,
we can take $t' \colon Q'_j \to N$ such that $t=t'f_j$.
Then we have $\iota t' f_j=rf_j$.

Now we consider the morphism 
$f':=(1_{Q'_j}-s \iota t')f_j \colon P(j) \to Q'_j$.
Since $N \subsetneq P(i)$, so $\iota$ is in the radical.
Thus $1_{Q'_j}-s \iota t'$ is an isomorphism,
so $f'$ is also a minimal left $(\add Q)$-approximation as well as $f_j$.
However, we have
\begin{align*}
rf'=r(1_{Q'_j}-s \iota t')f_j&=r(f_j-s \iota t'f_j)\\
&=r(f_j-srf_j)=rf_j-rsrf_j=rf_j-rf_j=0.
\end{align*}
This means that $\Im f'$ is contained in $\Ker r=Q'_j/{\Im s}$,
which is a proper direct summand of $Q'_j$.
Thus the idempotent endomorphism 
$\phi \colon Q'_j \to Q'_j/{\Im s} \to Q'_j$ satisfies
$\phi f'=f'$, but $\phi$ is not an automorphism.
Therefore $f'$ cannot be a minimal left $(\add Q)$-approximation,
a contradiction.

Therefore no $h \colon Q'_j \to M_i$ satisfies $h \ne 0$ and $hf_j=0$.
This means the injectivity.
\end{proof}

Thus we have the following important example.

\begin{Ex}\label{Ex_M_i}
Let $A=\bigoplus_{i=1}^n P(i)$ be basic, $X:=\SH(A)=\bigoplus_{i=1}^n L(i)$
and $J \subset \{1,\ldots,n\}$.
For each $i \notin J$, the module $M_i$ in Lemma \ref{Lem_M_i}
admits a short exact sequence
$0 \to M'_i \to M_i \to L(i) \to 0$ 
with $M'_i \in \Filt \{L(j) \mid j \in J\}$, 
which gives rise to a triangle
\begin{align*}
L(i)[-1] \to M'_i \to M_i \to L(i).
\end{align*}
The first map a minimal left $\calW_J$-approximation
by Lemmas \ref{Lem_M_i} (1) and \ref{Lem_ortho_0_1}.
Then $Y=\bigoplus_{i=1}^n Y_i:=\mu_J^-(X) \in \smc A$
is well-defined by Theorem \ref{Thm_smc_mut_approx}.
\end{Ex}

We would like to do the same thing in the abelian length category $\calH_S$ 
for $S \in \twosilt A$.
For this purpose, we rely on the following properties.
Note that $\add S$ may not be contained in $\calH_S$.

\begin{Prop}\label{Prop_E_S}
Let $S=\bigoplus_{i=1}^n S_i \in \silt A$, 
$E_S:=\End_{\sfD(A)}(S)$ and 
$X=\bigoplus_{i=1}^n \SH(S) \in \smc A$.
\begin{enumerate}
\item
\cite[Remark 3.1.17]{BBD}
For any $H_1,H_2 \in \calH_S$, we have an isomorphism
\begin{align*}
\Ext_{\calH_S}^1(H_1,H_2) \simeq \Hom_{\sfD(A)}(H_1,H_2[1]).
\end{align*}
\item
\cite[Theorem 4.8 and its proof]{IY}
The functor $\Psi_S:=\Hom_{\sfD(A)}(S,?)$ induces the equivalences
\begin{align*}
\Psi_S \colon \calH_S \to \mod E_S, \quad
\Psi_S \colon \add S \to \proj E_S.
\end{align*}
Moreover $\Psi_S$ gives the isomorphism
\begin{align*}
\Hom_{\sfD(A)}(S',H) \simeq \Hom_{E_S}(\Psi_S(S'),\Psi_S(H))
\end{align*}
functorial at $S' \in \add S$ and $H \in \calH_S$.
\item 
We have $\simple(\mod E_S)=\{\Psi_S(X_1),\ldots,\Psi_S(X_n)\}$.
Moreover for any $i \in \{1,\ldots,n\}$, 
the object $\Psi_S(S_i) \in \proj E_S$ is the projective cover of 
the simple $E_S$-module $\Psi_S(X_i)$.
\item
For each $i \in \{1,\ldots,n\}$,
we have an isomorphism 
\begin{align*}
\End_{\sfD(A)}(S_i)/{\rad \End_{\sfD(A)}(S_i)} \to \End_{\sfD(A)}(X_i)
\end{align*}
of algebras.
\end{enumerate}
\end{Prop}

\begin{proof}
(3)
By Proposition \ref{Prop_silt_smc},
we have $\simple \calH_S=\{X_1,\ldots,X_n\}$.
Thus the equivalence $\Psi_S \colon \calH_S \to \mod E_S$ in (2) 
gives $\simple (\mod E_S)=\{\Psi_S(X_1),\ldots,\Psi_S(X_n)\}$.

Next let $i \in \{1,\ldots,n\}$.
Then Definition-Proposition \ref{Def-Prop_SH} implies 
$\Hom_{\sfD(A)}(S_i,X_i) \ne 0$,
so (2) implies $\Hom_{E_S}(\Psi_S(S_i),\Psi_S(X_i)) \ne 0$ with
$\Psi_S(S_i)$ an indecomposable projective $E_S$-module and 
$X_i$ a simple $E_S$-module.
Thus $\Psi_S(S_i)$ is the projective cover of $\Psi_S(X_i)$.

(4)
Since $\Psi_S(S_i)$ is the projective cover of $\Psi_S(X_i)$ 
as shown in (3), 
we have 
\begin{align*}
\End_{E_S}(\Psi_S(S_i))/{\rad \End_{E_S}(\Psi_S(S_i))}
\simeq \End_{E_S}(\Psi_S(X_i)).
\end{align*}
Then the equivalences $\Psi_S \colon \add S \to \proj E_S$ and 
$\Psi_S \colon \calH_S \to \mod E_S$ in (2) induce the assertion.
\end{proof}

Combining Lemma \ref{Lem_M_i} and Proposition \ref{Prop_E_S},
we obtain the following properties.

\begin{Lem}\label{Lem_M_i_smc}
Let $S=\bigoplus_{i=1}^n S_i \in \silt A$, 
$X=\bigoplus_{i=1}^n X_i:=\SH(S) \in \smc A$,
$J \subset \{1,\ldots,n\}$, and
$\calW_J:=\Filt \{ X_j \mid j \in J \}\subset\sfD(A)$.
For each $i \in \{1,\ldots,n\} \setminus J$, the following hold.
\begin{enumerate}
\item
There exists a triangle 
\begin{align*}
X_i[-1] \xrightarrow{g_i} W_i \to Y_i \to X_i
\end{align*} 
in $\sfD(A)$ with $g_i$ a minimal left $\calW_J$-approximation.
In particular, we have $Y_i \in \calH_S$.
\item
Let $j \in J$,
and take a minimal left $(\add U)$-approximation $f_j \colon S_j \to U'_j$.
Then we have an isomorphism
\begin{align*}
\Hom_{\sfD(A)}(f_j,Y_i) \colon 
\Hom_{\sfD(A)}(U'_j,Y_i) \to \Hom_{\sfD(A)}(S_j,Y_i).
\end{align*}
\end{enumerate}
\end{Lem}

\begin{proof}
(1)
By Proposition \ref{Prop_E_S} (3), $\Psi_S(S_i)$ is the projective cover of 
the simple $E_S$-module $\Psi_S(X_i)$.
Under this context, 
we define $M_i \in \mod E_S$ as in Lemma \ref{Lem_M_i}.
We have a short exact sequence 
$0 \to M'_i \to M_i \to \Psi_S(X_i) \to 0$ in $\mod E_S$
with $M'_i \in \Filt_{E_S}\{\Psi_S(X_j) \mid j \in J\}$
by Lemma \ref{Lem_M_i} (1).
Then Proposition \ref{Prop_E_S} (2) gives a short exact sequence
$0 \to \Psi^{-1}(M'_i) \to \Psi^{-1}(M_i) \to X_i \to 0$ in $\calH_S$.
By this and Proposition \ref{Prop_E_S} (1),
we get a triangle 
$X_i[-1] \to \Psi^{-1}(M'_i) \to \Psi^{-1}(M_i) \to X_i$ in $\sfD(A)$.

Moreover Lemma \ref{Lem_M_i} (1) gives
$\Hom_{E_S}(M_i,\Psi_S(\calW_J))=0$ and $\Ext_{E_S}^1(M_i,\Psi_S(\calW_J))=0$.
Then Proposition \ref{Prop_E_S} (2) implies
$\Hom_{\calH_S}(Y_i,\calW_J)=0$ and $\Ext_{\calH_S}^1(Y_i,\calW_J)=0$.
By Proposition \ref{Prop_E_S} (1),
the latter means $\Hom_{\sfD(A)}(Y_i,\calW_J[1])=0$.
Thus the triangle
$X_i[-1] \to \Psi^{-1}(M'_i) \to \Psi^{-1}(M_i) \to X_i$ satisfies
the desired property by Lemma \ref{Lem_ortho_0_1}.

(2)
By Proposition \ref{Prop_E_S} (2), 
$\Psi(f_j) \colon \Psi(S_j) \to \Psi(U'_j)$ 
is a minimal left $(\add \Psi(U))$-approximation in $\mod E_S$. 
Then the assertion follows from Lemma \ref{Lem_M_i} (2) and 
Proposition \ref{Prop_E_S} (2).
\end{proof}

Then we already have Theorem \ref{Thm_smc_mut} (1).

\begin{proof}[Proof of Theorem \ref{Thm_smc_mut} (1)] 
This follows from Lemma \ref{Lem_M_i_smc} (1).
\end{proof}

In particular, $\mu_J^-(X)=\bigoplus_{i=1}^n Y_i$ is well-defined.
This and $T:=\mu_J^-(S)$ 
satisfy the same duality as Definition-Proposition \ref{Def-Prop_SH}.
The most crucial part is (ii-c).

\begin{Lem}\label{Lem_T_Y_dual}
Let $S \in \silt A$, $X=\bigoplus_{i=1}^n X_i:=\SH(S) \in \smc A$,
$J \subset \{1,\ldots,n\}$ and
$\calW_J:=\Filt \{ X_j \mid j \in J \}\subset\sfD(A)$.
Set $T:=\mu_J^-(S)$ and $Y:=\mu_J^-(X)$.
Then for any $i,i' \in \{1,\ldots,n\}$ and $l \in \Z$, 
we have
\begin{align*}
\Hom_{\sfD(A)}(T_i,Y_{i'}[l]) \simeq
\begin{cases}
\End_{\sfD(A)}(Y_i) \quad \text{as left $\End_{\sfD(A)}(Y_i)$-modules} 
& ((i',l)=(i,0)) \\
0 & ((i',l) \ne (i,0))
\end{cases}.
\end{align*}
\end{Lem}

\begin{proof}
We set $U:=S/S_J$ and take the triangle $S_j \to U'_j \to T_j \to S_j[1]$
for each $j \in J$ as in Definition-Proposition \ref{Def-Prop_silt_mut}.

(i)
Let $(i',l)=(i,0)$.
The assertion follows from (i-a) and (i-b) below.

(i-a)
Assume $i \notin J$. 
Then we have $T_i=U_i$.
Naturally, $\Hom_{\sfD(A)}(U_i,Y_i)$ is a left $\End_{\sfD(A)}(Y_i)$-module.

By Theorem \ref{Thm_smc_mut_approx}, 
we get $\End_{\sfD(A)}(Y_i) \simeq \End_{\sfD(A)}(X_i)$ as algebras.
In particular, $\End_{\sfD(A)}(Y_i)$ is a division $K$-algebra.
Thus, to show that 
$\Hom_{\sfD(A)}(U_i,Y_i)$ is a simple left $\End_{\sfD(A)}(Y_i)$-module,
it suffices to check
$\dim_K \Hom_{\sfD(A)}(U_i,Y_i)=\dim_K \End_{\sfD(A)}(Y_i)$.

By $W_i \in \calW_J$ and Definition-Proposition \ref{Def-Prop_SH},
$\Hom_{\sfD(A)}(U_i,W_i)=0$ and $\Hom_{\sfD(A)}(U_i,W_i[1])=0$ hold.
Then the triangle $X_i[-1] \to W_i \to Y_i \to X_i$ gives an isomorphism 
$\Hom_{\sfD(A)}(U_i,Y_i) \simeq \Hom_{\sfD(A)}(U_i,X_i)$
of $K$-vector spaces.
By Definition-Proposition \ref{Def-Prop_SH}, we get that 
$\Hom_{\sfD(A)}(U_i,X_i)$ is a simple left $\End_{\sfD(A)}(X_i)$-module.
Thus we get $\dim_K \Hom_{\sfD(A)}(U_i,X_i)=\dim_K \End_{\sfD(A)}(X_i)$.
Moreover we have seen $\End_{\sfD(A)}(X_i) \simeq \End_{\sfD(A)}(Y_i)$ above.

Therefore we get $\dim_K \Hom_{\sfD(A)}(U_i,Y_i)=\dim_K \End_{\sfD(A)}(Y_i)$
as desired.

(i-b)
Assume $i=:j \in J$; then $Y_j=X_j[1]$.
By the same reason as (i-a), it suffices to check
$\dim_K \Hom_{\sfD(A)}(T_j,Y_j)=\dim_K \End_{\sfD(A)}(Y_j)$.

By Definition-Proposition \ref{Def-Prop_SH},
we get $\Hom_{\sfD(A)}(U'_j,Y_j)=0$ and $\Hom_{\sfD(A)}(U'_j[1],Y_j)=0$.
Then the triangle $S_j \to U'_j \to T_j \to S_j[1]$ with $U'_j \in \add U$
gives that $\Hom_{\sfD(A)}(T_j,Y_j) \simeq \Hom_{\sfD(A)}(S_j[1],Y_j)$
as $K$-vector spaces.
Clearly, we have $\Hom_{\sfD(A)}(S_j[1],Y_j) \simeq \Hom_{\sfD(A)}(S_j,X_j)$.
The latter is a simple left $\End_{\sfD(A)}(X_j)$-module,
so we get $\dim_K \Hom_{\sfD(A)}(S_j,X_j)=\dim_K \End_{\sfD(A)}(X_j)$.
Moreover we have $\End_{\sfD(A)}(X_j) \simeq \End_{\sfD(A)}(Y_j)$
as division algebras.

Thus we get the assertion.

(ii)
Let $(i,l) \ne (i',0)$.
The assertion follows from (ii-a) to (ii-d) below.

(ii-a)
Let $i,i' \notin J$. 
Then we have $\Hom_{\sfD(A)}(T_i,Y_{i'}[l])=
\Hom_{\sfD(A)}(S_i,Y_{i'}[l])$.
By the triangle $X_{i'}[-1] \to W_{i'} \to Y_{i'} \to X_{i'}$,
it suffices to check $\Hom_{\sfD(A)}(S_i,W_{i'}[l])=0$
and $\Hom_{\sfD(A)}(S_i,X_{i'}[l])=0$.
The former is true by $i \notin J$, $W_{i'} \in \calW_J$ 
and Definition-Proposition \ref{Def-Prop_SH}.
The latter follows from 
$(i,l) \ne (i',0)$ and Definition-Proposition \ref{Def-Prop_SH}.

(ii-b)
Let $i \notin J$ and $i'=:j' \in J$. 
Then we get
$\Hom_{\sfD(A)}(T_i,Y_{i'}[l])=\Hom_{\sfD(A)}(S_i,X_{j'}[l+1])$.
This is zero by $i \ne j'$ and Definition-Proposition \ref{Def-Prop_SH}.

(ii-c)
If $i=:j \in J$ and $i' \notin J$, then we have an exact sequence
\begin{align*}
& \Hom_{\sfD(A)}(U'_j[1],Y_{i'}[l]) \to \Hom_{\sfD(A)}(S_j[1],Y_{i'}[l])
\to \Hom_{\sfD(A)}(T_j,Y_{i'}[l]) \to \\
& \Hom_{\sfD(A)}(U'_j,Y_{i'}[l])
\to \Hom_{\sfD(A)}(S_j,Y_{i'}[l]).
\end{align*}
Note that $Y_{i'} \in \calH_S$ 
by $i' \notin J$ and Lemma \ref{Lem_M_i_smc} (1).

We show that the first morphism is isomorphic.
If $l=1$, then this follows from Lemma \ref{Lem_M_i_smc} (2). 
If $l \ne 1$, this morphism is just $0 \to 0$ 
by $Y_{i'} \in \calH_S$ and Definition-Proposition \ref{Def-Prop_SH}.

This implies that the last morphism is also isomorphic.

Therefore we get $\Hom_{\sfD(A)}(T_j,Y_{i'}[l])=0$ from the exact sequence.

(ii-d)
If $i=:j \in J$ and $i'=:j' \in J$, 
then we have $\Hom_{\sfD(A)}(T_j,Y_{j'}[l])
=\Hom_{\sfD(A)}(T_j,X_{j'}[l+1])$ and an exact sequence
\begin{align*}
\Hom_{\sfD(A)}(S_j[1],X_{j'}[l+1]) \to \Hom_{\sfD(A)}(T_j,X_{j'}[l+1]) 
\to \Hom_{\sfD(A)}(U'_j,X_{j'}[l+1]).
\end{align*}
Thus it suffices to show $\Hom_{\sfD(A)}(S_j[1],X_{j'}[l+1])=0$ 
and $\Hom_{\sfD(A)}(U'_j,X_{j'}[l+1])=0$.
The former is true by $(j,l) \ne (j',0)$
and Definition-Proposition \ref{Def-Prop_SH}.
The latter holds by $U'_j \in \add U$ 
and Definition-Proposition \ref{Def-Prop_SH}.
\end{proof}

By using this, we can show the remaining part of 
Theorem \ref{Thm_smc_mut} as follows.

\begin{proof}[Proof of Theorem \ref{Thm_smc_mut} (2)]
We set $Z=\bigoplus_{i=1}^n Z_i:=\SH(T)$,
and show $Y_i \simeq Z_i$ for each $i \in \{1,\ldots,n\}$.

By Lemma \ref{Lem_T_Y_dual}, we have 
$\Hom_{\sfD(A)}(T,Y_i[l])=0$ for any $l \ne 0$,
so $Y_i \in \calH_T$ holds.

Lemma \ref{Lem_T_Y_dual} gives also $\Hom_{\sfD(A)}(T/T_i,Y_i)=0$,
so by Proposition \ref{Prop_E_S} (2),
any composition factor of $Y_i$ in $\calH_T$ is 
isomorphic to the simple object $Z_i$ in $\calH_T$
by Definition-Proposition \ref{Def-Prop_SH}. 

Moreover $\End_{\sfD(A)}(Y_i)$ is a division algebra
by Theorem \ref{Thm_smc_mut_approx}.

Thus we get $Y_i \simeq Z_i$ as desired.
\end{proof}

\section{2-term silting complexes and 2-term simple-minded collections}
\label{Sec_2-term}

In this section, we deal with properties of 
2-term silting complexes and 2-term simple-minded collections,
focusing on their relationship with the module category $\mod A$.

\subsection{Basic properties}

First, we consider 2-term (pre)silting complexes.

\begin{Def}
Let $U \in \sfK^\rmb(\proj A)$.
Then we say that $U$ is \emph{2-term} if the terms of $U$
except $-1$st and $0$th ones are zero.
We define $\twopsilt A$ (resp.~$\twosilt A$) for 
the set of 2-term presilting (resp.~ 2-term silting) complexes 
in $\sfK^\rmb(\proj A)$.
\end{Def}

2-term (pre)silting complexes satisfy the following important property.

\begin{Prop}\label{Prop_compl}
\cite[Proposition 2.16]{Ai}
Let $U \in \twopsilt A$.
We can complete $U$ to a basic 2-term silting complex; 
that is, there exists $S \in \twosilt A$ such that $U \in \add S$.
\end{Prop}

Thus $U \in \twosilt A$ is equivalent to that
both $U \in \twopsilt A$ and $|U|=|A|$ hold,
as in \cite[Proposition 3.3]{AIR}.

For any full subcategory $\calC \subset \mod A$, we define 
\begin{align*}
{^\perp \calC}&:=\{N \in \mod A \mid 
\text{$\Hom_A(N,M)=0$ for all $M \in \calC$}\},\\
{\calC^\perp}&:=\{N \in \mod A \mid 
\text{$\Hom_A(M,N)=0$ for all $M \in \calC$}\}.
\end{align*}
In the case $\calC=\{M\}$, they are simply written as
${^\perp M},{M^\perp}$.

A pair $(\calT,\calF)$ of full subcategories of $\mod A$
is called a \emph{torsion pair} in $\mod A$
if $\calF=\calT^\perp$ and $\calT={^\perp \calF}$.
A full subcategory $\calT$ is called a \emph{torsion class}
if there exists some $\calF \subset \mod A$ 
such that $(\calT,\calF)$ is a torsion pair in $\mod A$;
or equivalently, $\calT$ is closed under factor modules and extensions
in $\mod A$.
Dually, a full subcategory $\calF$ is called a \emph{torsion-free class}
if there exists some $\calT \subset \mod A$ 
such that $(\calT,\calF)$ is a torsion pair in $\mod A$;
or equivalently, $\calF$ is closed under submodules and extensions
in $\mod A$.

A full subcategory $\calC \subset \mod A$ is said to be 
\emph{functorially finite} if any module $M \in \mod A$
has a left $\calC$-approximation and a right $\calC$-approximation.
For a module $M \in \mod A$, we set
\begin{align*}
\Fac M&:=\{N \in \mod A \mid 
\text{there exists a surjection $M^{\oplus s} \to N$}\},\\
\Sub M&:=\{N \in \mod A \mid 
\text{there exists an injection $N \to M^{\oplus s}$}\}.
\end{align*}
Then with each $U \in \twopsilt A$,
we associate the following functorially finite torsion pairs.

\begin{Def-Lem}\label{Def-Lem_ftors}\cite{AS}
Let $U \in \twopsilt A$.
Then we define two torsion pairs 
$(\ovcalT_U,\calF_U)$ and $(\calT_U,\ovcalF_U)$ by
\begin{align*}
(\ovcalT_U,\calF_U)=({^\perp (H^{-1}(\nu U))},\Sub H^{-1}(\nu U)), \quad
(\calT_U,\ovcalF_U)=(\Fac H^0(U),H^0(U)^\perp).
\end{align*}
Then $\calT_U \subset \ovcalT_U$ and $\calF_U \subset \ovcalF_U$ hold,
and all of them are functorially finite.
\end{Def-Lem}

We remark that $H^0(U)$ is the $\tau$-rigid module
and $H^{-1}(\nu U)$ is the $\tau^{-1}$-rigid module 
associated with $U \in \twopsilt A$ in \cite{AIR}.
We do not recall their definitions explicitly in this paper.

The construction above gives all functorially finite torsion classes
and functorially finite torsion-free classes as follows.
We define $\ftors A$ (resp.~$\ftorf A$) 
as the set of functorially finite torsion classes 
(resp. functorially finite torsion-free classes) in $\mod A$.

\begin{Prop}\label{Prop_silt_tors}\cite[Theorems 2.7, 2.12, 3.2]{AIR}
We have the following maps.
\begin{enumerate}
\item
There exist four surjections
\begin{alignat*}{7}
\twopsilt A &\to \ftors A, &\quad&& U &\mapsto \ovcalT_U; &\quad&&
\twopsilt A &\to \ftorf A, &\quad&& U &\mapsto \calF_U; \\
\twopsilt A &\to \ftors A, &\quad&& U &\mapsto \calT_U; &\quad&&
\twopsilt A &\to \ftorf A, &\quad&& U &\mapsto \ovcalF_U.
\end{alignat*}
\item
There exist bijections
\begin{alignat*}{7}
\twosilt A &\to \ftors A, &\quad&& S &\mapsto \ovcalT_S=\calT_S; &\quad&&
\twosilt A &\to \ftorf A, &\quad&& S &\mapsto \calF_S=\ovcalF_S.
\end{alignat*}
\item
Let $U \in \twopsilt A$.
Then $U \in \twosilt A$ is equivalent to 
$(\ovcalT_U,\calF_U)=(\calT_U,\ovcalF_U)$.
\end{enumerate}
\end{Prop}

Next we consider 2-term simple-minded collections.

\begin{Def}
Let $X \in \twosmc A$.
Then $X$ is said to be \emph{2-term} if $H^l(X)=0$ 
for any $l \in \Z \setminus \{-1,0\}$.
We write $\twosmc A$ for the set of isoclasses of 
2-term simple-minded collections in $\sfD(A)$.
\end{Def}

Then we have the following result.

\begin{Prop}\label{Prop_silt_smc_2}\cite[Corollary 4.3]{BY}
The bijection $\silt A \to \smc A$ in Proposition \ref{Prop_silt_smc}
is restricted to a bijection $\twosilt A \to \twosmc A$.
\end{Prop}

We recall a fundamental property of 2-term simple-minded collections.

\begin{Prop}\label{Prop_stalk}
\cite[Remark 4.11]{BY}
Let $X=\bigoplus_{i=1}^n X_i \in \twosmc A$.
For each $i \in \{1,\ldots,n\}$, 
the complex $X_i \in \sfD(A)$ belongs to $\mod A$ or $\mod A[1]$.
\end{Prop}

Therefore we define the next symbols.

\begin{Nota}
Let $X=\bigoplus_{i=1}^n X_i \in \twosmc A$.
For each $i \in \{1,\ldots,n\}$, 
we define $X_i^+:=H^0(X_i)$ and $X_i^-:=H^{-1}(X_i)$.
Moreover we set $X^+:=\bigoplus_{i=1}^n X_i^+=H^0(X)$ and
$X^-:=\bigoplus_{i=1}^n X_i^-=H^{-1}(X)$.
Then we have $X=X^+ \oplus X^-[1]$.
\end{Nota}

\subsection{Semibricks}

To study 2-term simple-minded collections,
we use semibricks and related notions.
In (3), $\sfT(M)$ (resp.~$\sfF(M)$) denotes
the smallest torsion class (resp.~the smallest torsion-free class)
containing $M$.

\begin{Def}
Let $M \in \mod A$.
\begin{enumerate}
\item
We say that $M$ is a \emph{brick} in $\mod A$
if $\End_A(M)$ is a division $K$-algebra.
\item
\cite[Definition 2.1]{A-semi}
Decompose $M$ as $M=\bigoplus_{i=1}^m M_i$ with $M_i$ indecomposable. 
Then we say that $M$ is a \emph{semibrick} in $\mod A$ if 
\begin{enumerate2}
\item
for any $i \in \{1,\ldots,m\}$, the module $M_i$ is a brick; and
\item
any $i \ne j \in \{1,\ldots,m\}$ satisfy $\Hom_A(M_i,M_j)=0$.
\end{enumerate2}
\item
\cite[Definition 2.2]{A-semi}
Let $M$ be a semibrick in $\mod A$.
Then $M$ is said to be \emph{left finite} (resp. \emph{right finite})
if $\sfT(M)$ (resp. $\sfF(M)$) is functorially finite.
We write $\fLsbrick A$ (resp.~$\fRsbrick A$) for 
the set of left finite (resp.~right finite) semibricks.
\end{enumerate}
\end{Def}

In particular, if $M$ is a semibrick, then $M$ is basic.
For a semibrick $M=\bigoplus_{i=1}^m M_i$, 
its \emph{size} is defined as the number $|M|=m$.
We have the following property.

\begin{Prop}\label{Prop_LF_bound}\cite[Corollary 2.10]{A-semi}
If a semibrick $M$ is left finite or right finite, 
then we have $|M| \le |A|$.
\end{Prop}

We remark that left finiteness and right finiteness of semibricks
do not coincide in general.
Moreover left finiteness is not necessarily preserved
under taking direct summands;
in other words, if $M$ is a left finite semibrick and 
$N$ is a direct summand of $M$, then $N$ is always a semibrick,
but $N$ may not be left finite.
In \cite[Subsection 4.2]{A-semi}, we gave such examples.

On the other hand, for a semibrick $M$,
it is a direct summand of some left finite semibrick $N$
if and only if $M$ is a direct summand of some right finite semibrick $N'$
as shown in Corollary \ref{Cor_LF_RF_summand}.

By Proposition \ref{Prop_stalk}, we immediately have the following properties.

\begin{Lem}\label{Lem_LF_size}
Let $X=\bigoplus_{i=1}^n \in \twosmc A$.
\begin{enumerate}
\item
For each $i \in \{1,\ldots,n\}$,
one of $X_i^+$ and $X_i^-$ is a brick, and the other is zero.
\item
The module $X^+$ (resp.~$X^-$) is a semibrick,
and its size $|X^+|$ (resp.~$|X^-|$) is 
the number of elements $i \in \{1,\ldots,n\}$
such that $X_i^+ \ne 0$ (resp.~$X_i^- \ne 0$).
\end{enumerate}
\end{Lem}

Actually, $X^+$ is left finite and $X^-$ is right finite,
and these correspondences give bijections between
2-term simple-minded collections, 
left finite semibricks and right finite semibricks.

\begin{Prop}\label{Prop_sbrick}
\cite[Theorem 3.3]{A-semi}
The following statements hold.
\begin{enumerate}
\item 
We have bijections
\begin{align*}
\twosmc A \to \fLsbrick A \to \ftors A \quad \text{and} \quad
\twosmc A \to \fRsbrick A \to \ftorf A
\end{align*}
respectively given by 
\begin{align*}
X \mapsto X^+, \quad M \mapsto \sfT(M), \quad \text{and} \quad 
X \mapsto X^-, \quad M \mapsto \sfF(M).
\end{align*}
\item
Let $X \in \twosmc A$ and $S \in \twosilt A$ satisfy $X=\SH(S)$. 
Then we have 
\begin{align*}
(\ovcalT_S,\calF_S)=(\calT_S,\ovcalF_S)=(\sfT(X^+),\sfF(X^-)).
\end{align*}
\end{enumerate}
\end{Prop}

As a direct consequence, we get the following property.

\begin{Lem}\label{Lem_X^+=X}
If $X \in \twosmc A$ satisfies $X^+=X$, 
then $X$ is the direct sum $\bigoplus_{i=1}^n L(i)$ of simple $A$-modules.
\end{Lem}

\begin{proof}
Let $X \in \twosmc A$ with $X^+=X$.
Then we have $X^-=0$ and $\sfF(X^-)=\{0\}$.
Thus such $X$ is unique by Proposition \ref{Prop_sbrick}.
Clearly, if $X=\bigoplus_{i=1}^n L(i)$, then $X \in \twosmc A$ with $X^+=X$.
\end{proof}

We have seen bijections between the sets
$\twosilt A,\twosmc A,\ftors A,\fLsbrick A$ 
in Propositions \ref{Prop_silt_tors} and \ref{Prop_sbrick}.
These sets can be finite or infinite.
There is a nice result on this.

\begin{Def-Prop}\cite{DIJ,Sentieri,AnS}
A finite dimensional $K$-algebra $A$ is said to be \emph{brick finite}
if it satisfies the following equivalent conditions.
\begin{enumerate2}
\item
There exist only finitely many isoclasses of bricks.
\item
The sets $\twosilt A,\twosmc A,\ftors A,\fLsbrick A$ are finite.
\item
Every torsion class $\calT \subset \mod A$ is functorially finite.
\item
Every torsion class $\calT \subset \mod A$ has a semibrick $M$
such that $\calT=\sfT(M)$.
\end{enumerate2}
\end{Def-Prop}

\begin{proof}
(a), (b) and (c) are equivalent by \cite[Theorems 3.8, 4.2]{DIJ}.

(c)$\Leftrightarrow$(d) is \cite[Corollary 4.10]{Sentieri}
and \cite[Proposition 5.4]{AnS}.
\end{proof}

Therefore if $A$ is brick finite, then every semibrick is left finite.
The converse is still an open problem.

\begin{Conj}
Let $A$ be a finite dimensional $K$-algebra.
If every semibrick is left finite, then $A$ is brick finite.
\end{Conj}

In view of Proposition \ref{Prop_LF_bound},
our conjecture follows from Enomoto conjecture 
\cite[Conjecture 5.12]{Enomoto}, which claims that,
if the size of every semibrick in $A$ is finite,
then $A$ is brick finite.
Such observations are extended in \cite[Corollary 3.3]{MP} to 
the size of sets of Hom-orthogonal modules
which are not necessarily bricks.
There are many other interesting conjectures on brick finiteness 
\cite{A-bicpt,Demonet,Pfeifer,STV}.

We next explain how to directly calculate 
the corresponding 2-term simple-minded collection $X=\SH(S) \in \twosmc A$ 
for a given 2-term silting complex $S \in \twosilt A$ 
in Proposition \ref{Prop_silt_smc_2}.

\begin{Prop}\label{Prop_endotop}
\cite[Lemma 3.14]{A-semi}
Let $S=\bigoplus_{i=1}^n S_i \in \twosilt A$,
and $X=\bigoplus_{i=1}^n X_i:=\SH(S) \in \twosmc A$.
For each $i \in \{1,\ldots,n\}$, we have
\begin{align*}
X_i^+=H^0(S_i)/\sum_{f \in \rad_A(H^0(S),H^0(S_i))} \Im f, \quad 
X_i^-=\bigcap_{f \in \rad_A(H^{-1}(\nu S_i),H^{-1}(\nu S))} \Ker f.
\end{align*}
Therefore the semibricks $X^+$ and $X^-$ are given by
\begin{align*}
X^+&=H^0(S)/\sum_{f \in \rad \End_A(H^0(S))} \Im f
=H^0(S)/\rad_{\End_A(H^0(S))}H^0(S), \\
X^-&=\bigcap_{f \in \End_A(H^{-1}(\nu S))} \Ker f
=\soc_{\End_A(H^{-1}(\nu S))}H^{-1}(\nu S).
\end{align*}
\end{Prop}

\subsection{The main results and applications}

For a 2-term silting complex $S \in \twosilt A$,
the left mutation $\mu_J^-(S)$ at each direct summand $S_J$ may not be 2-term.
We prove the following characterization when
$\mu_J^-(S)$ is again 2-term.

\begin{Thm}\label{Thm_within}
Let $S=\bigoplus_{i=1}^n S_i \in \twosilt A$,
$X=\bigoplus_{i=1}^n X_i := \SH(S) \in \twosmc A$ and
$J \subset \{1,\ldots,n\}$.
Set $U:=S/S_J$, $T:=\mu_J^-(S)$ and $Y:=\mu_J^-(X)$.
Then the following conditions are equivalent.
\begin{enumerate2}
\item
The silting complex $T$ satisfies $T \in \twosilt A$.
\item
The simple-minded collection $Y$ satisfies $Y \in \twosmc A$.
\item
For any $j \in J$, we have $X_j^+ \ne 0$.
\item
The torsion pair $(\ovcalT_S,\calF_S)$ coincide with $(\ovcalT_U,\calF_U)$.
\end{enumerate2}
In this case, we have 
\begin{align*}
\sfT(Y^+)=\calT_T=\calT_U \subset \ovcalT_U=\ovcalT_S=\sfT(X^+), \quad
\sfF(Y^-)=\ovcalF_T=\ovcalF_U \supset \calT_U=\calT_S=\sfF(X^-),
\end{align*}
where $Y:=\SH(T)=\mu_J^-(X)$.
\end{Thm}

\begin{proof}
Note that $Y=\SH(T)$ by Theorem \ref{Thm_smc_mut}.

(a)$\Leftrightarrow$(b) is clear by Proposition \ref{Prop_silt_smc_2}.

(b)$\Rightarrow$(c)
If (c) does not hold, then Proposition \ref{Prop_stalk} implies that
there exists some $j \in J$ such that $X_j \in (\mod A)[1]$.
Then Theorem \ref{Thm_smc_mut} gives $Y_j=X_j[1] \in (\mod A)[2]$,
so $Y \notin \twosmc (A)$.

(c)$\Rightarrow$(b)
By Theorem \ref{Thm_smc_mut} and (c), 
$Y_j=X_j[1] \in (\mod A)[1]$ for any $j \in J$. 
Let $i \in \{1,\ldots,n\} \setminus J$.
We take the triangle $X_i[-1] \to W_i \to Y_i \to X_i$ 
in Theorem \ref{Thm_smc_mut}.
For any $l \in \Z \setminus \{-1,0\}$,
we have $H^l(W_i)=0$ and $H^l(X_i)=0$ by $X \in \twosmc A$,
so $H^l(Y_i)=0$ also holds.
Thus $Y \in \twosmc A$.

(c)$\Rightarrow$(d)
By Proposition \ref{Prop_sbrick} (2), we have $\calF_S=\sfF(X^-)$.
On the other hand, we get $X_j^-=0$ for all $j \in J$ 
by (c) and Proposition \ref{Prop_stalk}.
Thus $X^-=\bigoplus_{i \notin J} X_i^-$ holds, 
and it is a submodule of 
$\bigoplus_{i \notin J} H^{-1}(\nu S_i)=H^{-1}(\nu U) \in \calF_U$
by Proposition \ref{Prop_endotop}.
Therefore $\calF_S=\sfF(X^-) \subset \calF_U$.
The converse $\calF_U \subset \calF_S$ is obvious.

(d)$\Rightarrow$(c)
Let $j \in J$. 
We have $H^{-1}(\nu S_j) \in \calF_S=\calF_U=\Sub H^{-1}(\nu U)$,
where (d) is used for $\calF_S=\calF_U$.
In particular, there exists a monomorphism 
$H^{-1}(\nu S_j) \to (H^{-1}(\nu U))^{\oplus s}$, 
which is in the radical by $U=\bigoplus_{i \notin J}U_i$.
Thus Proposition \ref{Prop_endotop} implies $X_j^-=0$,
and then Proposition \ref{Prop_stalk} give (c).

Now the equivalence of (a)--(d) is shown.
It remains to show the last statements.
We only check the first half.
Proposition \ref{Prop_sbrick} (2) gives
$\sfT(Y)=\ovcalT_T$ and $\ovcalT_S=\sfT(X)$.
By (d) and its dual, $\calT_T=\calT_U$ and $\ovcalT_U=\ovcalT_S$ hold.
Moreover Definition-Lemma \ref{Def-Lem_ftors} implies 
$\calT_U \subset \ovcalT_U$. 
\end{proof}

We remark that this theorem for the case $\#J=1$ follows from
\cite[Proposition 2.13, Theorem 3.12]{A-semi} and
\cite[Definition-Proposition 2.28, Corollary 3.8]{AIR},
which are fundamental to define the brick labeling
of the exchange quiver $Q(\twosilt A)$ in Subsection \ref{Subsec_exch}.

Next we observe how a mutation within $\twosmc A$  
changes the sizes of the left finite semibricks.

\begin{Prop}\label{Prop_X^-_Y^+}
Let $X=\bigoplus_{i=1}^n X_i \in \twosmc A$
and $J \subset \{ i \in \{1,\ldots,n\} \mid X_i^+ \ne 0\}$.
Then $Y=\bigoplus_{i=1}^n Y_i:=\mu_J^-(X) \in \twosmc A$ holds,
and it satisfies the following properties.
\begin{enumerate}
\item
For each $i \in \{1,\ldots,n\} \setminus J$,
at least one of $Y_i^+ \ne 0$ and $X_i^- \ne 0$ holds.
\item
For each $j \in J$, we have $X_i^+=Y_i^- \ne 0$ and $X_i^-=Y_i^+=0$.
\item
We have $|Y^+| \ge |X^+|-\#J$.
\end{enumerate}
\end{Prop}

\begin{proof}
First, $Y \in \twosmc A$ is assured by Theorem \ref{Thm_within}.

(1)
We take the triangle $X_i[-1] \to W_i \to Y_i \to X_i$
with the first map a minimal left $\calW_J$-approximation.
If both $Y_i^+=0$ and $X_i^-=0$ hold,
then we have $Y_i \in \mod A[1]$ and $X_i \in \mod A$,
so $\Hom_A(Y_i,X_i)=0$.
Thus the triangle splits, and we get $W_i \simeq X_i[1] \oplus Y_i$.
Since $W_i \in \calW_J \subset \mod A$ by the assumption on $J$, 
we get $Y_i \in \mod A$, which contradicts $Y_i^+=0$.

(2) is clear, because $X_j,Y_j \in (\mod A) \cup (\mod A)[1]$ 
and $Y_j=X_j[1]$ imply $X_j \in \mod A$ and $Y_j \in (\mod A)[1]$.

(3)
We set $I_X:=\{i \in \{1,\ldots,n\} \mid X_i^+ \ne 0\}$
and $I_Y:=\{i \in \{1,\ldots,n\} \mid Y_i^+ \ne 0\}$.
By (1), we have $I_Y \setminus J \supset I_X \setminus J$.
Thus we get 
$\#I_Y \ge \#(I_Y \setminus J) \ge \#(I_X \setminus J) \ge \#I_X-\#J$.
This and Lemma \ref{Lem_LF_size} (2) give $|Y^+| \ge |X^+|-\#J$.
\end{proof}

The inequality $|Y^+| \le |X^-|+\#J$ does not hold in general.  

\begin{Ex}
Let $A$ be the path algebra of the quiver
\begin{tikzpicture}[->,baseline=0pt,scale=0.5]
\node (1) at ( 0, 1) {$1$};
\node (2) at (-1,-1) {$2$};
\node (3) at ( 0,-1) {$3$};
\node (4) at ( 1,-1) {$4$};
\draw (1) to (2);
\draw (1) to (3);
\draw (1) to (4);
\end{tikzpicture}.
One can check that
\begin{align*}
X:=\begin{smallmatrix}
&1&\\2&3&4
\end{smallmatrix} \oplus \begin{smallmatrix}
2
\end{smallmatrix}[1] \oplus \begin{smallmatrix}
3
\end{smallmatrix}[1] \oplus \begin{smallmatrix}
4
\end{smallmatrix}[1] \in \twosmc A.
\end{align*}
Clearly, we get $|X^+|=1$.
Its left mutation $Y$ at $\begin{smallmatrix}
&1&\\2&3&4
\end{smallmatrix}$ (hence $\#J=1$) is 
\begin{align*}
Y=\begin{smallmatrix}
&1&\\2&3&4
\end{smallmatrix}[1] \oplus \begin{smallmatrix}
&1&\\&3&4
\end{smallmatrix} \oplus \begin{smallmatrix}
&1&\\2&&4
\end{smallmatrix} \oplus \begin{smallmatrix}
&1&\\2&3&
\end{smallmatrix}.
\end{align*}
Thus we get $|Y^+|=3$,
so $|Y^+| \ge |X^+|-\#J$ is true, but $|Y^+| \le |X^+|+\#J$ fails.
\end{Ex}

As an application of mutations within $\twosmc A$,
one can show the following property on maximal semibricks.
Here, we say that a semibrick $M$ is \emph{maximal}
if it is not a proper direct summand of another semibrick.

\begin{Prop}\label{Prop_LF_RF_max}
(cf.~\cite[Proposition 5.6, Corollary 5.7]{BDH})
Let $M \in \fRsbrick A$.
Take the unique $Y=\bigoplus_{i=1}^n Y_i \in \twosmc A$ such that $Y^-=M$,
and set $J:=\{i \in \{1,\ldots,n\} \mid Y_i^- \ne 0\}$.
Then $X:=\mu_J^+(Y) \in \twosmc A$ holds,
and $X^+$ is a left finite, right finite and maximal semibrick
which contains $M$ as a direct summand.
\end{Prop}

\begin{proof}
Though the original argument in \cite{BDH} works,
we rewrite it in our setting for the convenience of the reader.

By Theorem \ref{Thm_within} and the definition of $J$,
we have $X=\mu_J^+(Y) \in \twosmc A$.
Then $X^+$ is a left finite semibrick by Proposition \ref{Prop_sbrick}.
By Theorem \ref{Thm_smc_mut}, $M=Y^-$ is a direct summand of $X^+$.

We set $J':=\{i \in \{1,\ldots,n\} \mid X_i^+ \ne 0\}$.
Consider the left mutation $Z:=\mu_{J'}^-(X)$,
which is in $\twosmc A$ again by Theorem \ref{Thm_within}.
Set $J'':=\{i \in \{1,\ldots,n\} \mid Z_i^- \ne 0\}$.
Immediately, we have $J \subset J' \subset J''$ by Theorem \ref{Thm_smc_mut}.
We claim $J'=J''$, so let $i \in J'' \setminus J'$.
 
Since $i \notin J$ and $Y=\mu_J^-(X)$, 
there exists a triangle $X_i[-1] \xrightarrow{\phi} W_i \to Y_i \to X_i$
in $\sfD(A)$ with $\phi$ a minimal left $\calW_J$-approximation,
where $\calW_J:=\Filt \{X_i \mid i \in J\} \subset \calD(A)$.
Similarly, since $i \notin J'$ and $Z=\mu_{J'}^-(X)$, we take a triangle 
$X_i[-1] \xrightarrow{\phi'} W'_i \to Z_i \to X_i$
in $\sfD(A)$ with $\phi'$ a minimal left $\calW_{J'}$-approximation
with $\calW_{J'}:=\Filt \{X_i \mid i \in J'\}$.

By $J \subset J'$ and the definition of $J'$, 
we have $\calW_J \subset \calW_{J'} \subset \mod A$.
Moreover $i \notin J'$ implies $X_i \in (\mod A)[1]$.
Thus $\phi \colon X_i[-1] \to W_i$ and $\phi' \colon X_i[-1] \to W'_i$
are just morphisms in $\mod A$.
In addition, $i \notin J$ and $i \in J''$ imply 
$Y_i \in \mod A$ and $Z_i \in (\mod A)[1]$.
Thus we have short exact sequences
$0 \to X_i[-1] \xrightarrow{\phi} W_i \to Y_i \to 0$ and
$0 \to Z_i[-1] \to X_i[-1] \xrightarrow{\phi'} W'_i \to 0$ in $\mod A$. 
Since $\phi'$ is a minimal left $\calW_{J'}$-approximation
and $W_i \in \calW_J \subset \calW_{J'}$,
there exists $\psi \colon W'_i \to W_i$ such that $\phi=\psi\phi'$.
Since $\phi$ is injective, so is $\phi'$.
Thus $\phi'$ is isomorphic, a contradiction to $Z_i[-1] \ne 0$.

Therefore we have $J'=J''$.
This gives $X^+=Z^-$, so $X^+$ is right finite.

We finally show that $X^+$ is maximal. 
By the definition of semibricks,
it is enough to prove $(X^+)^\perp \subset \sfF(X^+)$.
By Proposition \ref{Prop_sbrick}, 
we have $(\sfT(X^+),\sfF(X^-))$ is a torsion pair,
which implies $(X^+)^\perp=\sfF(X^-)$.
By Theorem \ref{Thm_within} and $Z=\mu_{J'}^-(X)$, 
we get $\sfF(X^-) \subset \sfF(Z^-)$.
Clearly, $Z^-=X^+$ gives $\sfF(Z^-)=\sfF(X^+)$.
Therefore $(X^+)^\perp \subset \sfF(X^+)$ holds as desired.
\end{proof}

We remark that the original results in \cite{BDH}
were stated and proved in the context of the pop-stack operators 
on the lattice $\tors A$ of torsion classes in the case $A$ is brick finite.
In general, $\ftors A$ is not a sublattice of $\tors A$,
but the pop-stack operators are restricted on $\ftors A$,
since they are compatible with certain mutations in $\twosmc A$ 
by the same argument as \cite[Theorem 5.1]{BDH}.
In our proof above, $\sfT(X^+)$ is the pop of $\sfT(Y^+)$,
and $\sfT(Z^+)$ is the stack of $\sfT(X^+)$.
A more combinatorial proof in the finite case is available 
at \cite[Remark 9.9]{DW}.

As a direct consequence, we have the following result.

\begin{Cor}\label{Cor_LF_RF_summand}
Let $M$ be a semibrick.
Then the following conditions are equivalent.
\begin{enumerate2}
\item
There exists $N \in \fLsbrick A$ such that $M \in \add N$.
\item
There exists $N \in \fRsbrick A$ such that $M \in \add N$.
\item
There exists $N \in \fLsbrick A \cap \fRsbrick A$ such that $M \in \add N$.
\end{enumerate2}
\end{Cor}

\begin{proof}
By symmetry, we only show (b)$\Leftrightarrow$(c).
(c)$\Rightarrow$(b) is obvious.
To show (b)$\Rightarrow$(c),
we may assume that $M$ itself is right finite.
Then Proposition \ref{Prop_LF_RF_max} gives a semibrick in (c).
\end{proof}

\subsection{The exchange quiver}\label{Subsec_exch}

In this subsection, 
we only consider irreducible mutations within $\twosilt A$ and $\twosmc A$.
As in \cite{AIR} (cf.~\cite[Proposition 5.7]{DF}),
for each indecomposable direct summand $S_j$ of $S \in \twosilt A$,
exactly one of $\mu_j^-(S)$ and $\mu_j^+(S)$ is in $\twosilt A$,
and which is $\twosilt A$ is determined 
by inclusions of the corresponding torsion classes.
We add some additional information on this.

\begin{Prop}\label{Prop_two_compl}
\cite[Definition-Proposition 2.28, Corollary 3.8]{AIR}
Let $S=\bigoplus_{i=1}^n S_i \in \twosilt A$,
and $j \in \{1,\ldots,m\}$.
\begin{enumerate}
\item
There uniquely exists $T=\bigoplus_{i=1}^n T_i \in \twosilt A$ 
such that $S_j \not \simeq T_j$ and $S_i=T_i$ for all $i \ne j$.
\item
In (1), exactly one of the following is true,
where $X=\bigoplus_{i=1}^n X_i := \SH(S) \in \twosmc A$.
\begin{enumerate2}
\item
The conditions
$\calT_T=\calT_{S/S_j} \subsetneq \calT_S$, 
$\calF_T \supsetneq \calF_{S/S_j}=\calF_S$,
$T=\mu_j^-(S)$ and $X_j^+ \ne 0$ hold.
\item
The conditions
$\calT_T \supsetneq \calT_{S/S_j}=\calT_S$,
$\calF_T=\calF_{S/S_j} \subsetneq \calF_S$, 
$T=\mu_j^+(S)$ and $X_j^- \ne 0$ hold.
\end{enumerate2} 
\end{enumerate}
\end{Prop}

\begin{proof}
In view of the cited results, it remains to show
$T=\mu_j^-(S)$, $X_j^+ \ne 0$ in (2)(a)
and $T=\mu_j^+(S)$, $X_j^- \ne 0$ in (2)(b).
In (2)(a), 
$\calF_S=\calF_{S/S_j}$ implies $H^{-1}(\nu S_j) \in \calF_{S/S_j}$,
so we get $\mu_j^-(S) \in \twosilt A$ and $X_j^+ \ne 0$
by Theorem \ref{Thm_within},
and then the uniqueness in (1) gives $T=\mu_j^-(S)$.
Similar argument works in (b).
\end{proof}

Irreducible mutations of 2-term silting complexes define a quiver.
The arrows of the quiver are labeled by bricks as follows.

\begin{Def}\label{Def_label}
We define the following notions.
\begin{enumerate}
\item
The \emph{exchange quiver} $Q(\twosilt A)$ of $\twosilt A$ is set as 
the following quiver.
\begin{enumerate}
\item
The vertices are elements of $\twosilt A$.
\item
For $S,T \in \twosilt A$,
there is an arrow if and only if 
$T$ is a left irreducible mutation of $S$.
\end{enumerate}
\item
\cite[Definitions 2.14, 3.10, Theorem 3.12]{A-semi}
Let $\alpha \colon S \to T$ be an arrow in $Q(\twosilt A)$.
Decompose $S=\bigoplus_{i=1}^n S_i$ with $S_i$ indecomposable,
and take the unique $j \in \{1,\ldots,n\}$ such that $T=\mu_j^-(S)$.
Then we call the brick $Z_\alpha:=X_j^+$ 
the \emph{label} of the arrow $S \to T$.
This association $\alpha \mapsto Z_\alpha$ is called 
the \emph{brick labeling} of $Q(\twosilt A)$.
\end{enumerate}
\end{Def}

By Proposition \ref{Prop_two_compl}, 
the exchange quiver $Q(\twosilt A)$ is \emph{$n$-regular};
that is, each vertex has exactly $n$ arrows starting or ending at it.
Moreover, the following properties hold by definition.

\begin{Prop}\label{Prop_smc_label}
Let $S \in \twosilt A$ and $X:=\SH(S) \in \twosmc A$.
\begin{enumerate}
\item
The object $X^+$ is the direct sum of 
the labels of the arrows starting at $S$ in $Q(\twosilt A)$.
\item
The object $X^-$ is the direct sum of 
the labels of the arrows ending at $S$ in $Q(\twosilt A)$.
\end{enumerate}
\end{Prop}

We give an extreme example.

\begin{Ex}\label{Ex_A_source}
Assume that $A$ is basic with $A=\bigoplus_{i=1}^n P(i)$.
Then $A \in \twosilt A$ holds, and $\sfT(A)=\mod A$.
Thus Proposition \ref{Prop_two_compl} implies that 
the vertex $A$ is a source in $Q(\twosilt A)$.
The arrows starting at $A$ are 
$\alpha_i \colon A \to \mu_i^-(A)$ for $i \in \{1,\ldots,n\}$,
and each $\alpha_i$ is labeled by the simple module $L(i)$,
since $\SH(A)=\bigoplus_{i=1}^n L(i)$ as in Example \ref{Ex_SH(A)}.

Moreover $A$ is the unique source in $Q(\twosilt A)$;
indeed, if $S$ is a source, then $X:=\SH(S)$ satisfies $X^+=X$ and $X^-=0$
by Proposition \ref{Prop_smc_label},
so we get $X=\bigoplus_{i=1}^n L(i)$ by Lemma \ref{Lem_X^+=X} and $S=A$.
\end{Ex}

We remark the second paragraph also follows from \cite[Theorem 3.1]{DIJ}.

Each arrow of $Q(\twosilt A)$ and its label satisfy the following.

\begin{Prop}\label{Prop_label_filt}
\cite[Theorem 2.18]{AIR} \cite[Proposition 2.17]{A-semi}
Assume that $S \to T$ is an arrow in $Q(\twosilt A)$, 
and let $Z \in \brick A$ be its label.
Then the maximal common direct summand $U \in \twopsilt A$ of $S,T$
satisfies
\begin{align*}
|U|=|A|-1, \quad 
(\ovcalT_U,\calF_U)=(\ovcalT_S,\calF_S),\quad 
(\calT_U,\ovcalF_U)=(\calT_T,\ovcalF_T),\quad
\ovcalT_U \cap \ovcalF_U=\Filt Z.
\end{align*}
\end{Prop}

We remark that the inclusions of torsion classes 
appearing Proposition \ref{Prop_two_compl} are adjacent
in the lattice $\tors A$ of all torsion classes \cite[Example 3.5]{DIJ}.
Thus $Q(\twosilt A)$ is a full subquiver of the Hasse quiver 
$\Hasse(\tors A)$.
Our brick labeling in $Q(\twosilt A)$ is extended to $\Hasse(\tors A)$
in \cite[Definition 2.10]{BCZ} and \cite[Definition 3.5]{DIRRT}; 
compare them to Proposition \ref{Prop_label_filt}.

\section{Maximal/minimal completions and $\tau$-tilting reduction}
\label{Sec_reduc}

\subsection{Maximal/minimal completions}

In this subsection,
we consider the following two completions of each $U \in \twopsilt A$.

\begin{Def-Prop}
Let $U \in \twopsilt A$.
By Proposition \ref{Prop_silt_tors}, 
there uniquely exist $S,T \in \twosilt A$ such that
\begin{align*}
(\ovcalT_U,\calF_U)=(\ovcalT_S,\calF_S), \quad
(\calT_U,\ovcalF_U)=(\calT_T,\ovcalF_T).
\end{align*}
Then $S$ and $T$ are called the \emph{maximal completion}
and the \emph{minimal completion} of $U$.
\end{Def-Prop}

Sometimes $S$ and $T$ are called the \emph{Bongartz completion}
and the \emph{Bongartz co-completion} of $U$.
The reason why we say $S$ and $T$ are maximal and minimal is 
the following property;
note that $\calT_V=\ovcalT_V$ always holds since $V \in \twosilt A$.

\begin{Prop}\cite[Proposition 2.9]{AIR}
Let $U \in \twopsilt A$ and $V \in \twosilt A$.
Then $U \in \add V$ holds if and only if 
$\calT_U \subset \calT_V \subset \ovcalT_U$.
\end{Prop}

The following observations are natural but important in our arguments.

\begin{Lem}\label{Lem_S_T_common}
Let $U \in \twopsilt A$ and $S,T \in \twosilt A$.
Then the following conditions are equivalent.
\begin{enumerate2}
\item
The complex $S$ is the maximal completion of $U$,
and $T$ is the minimal completion of $U$.
\item
The complex $U$ is a direct summand of $S$,
and the left mutation of $S$ at $S/U$ is $T$. 
\end{enumerate2}
\end{Lem}

\begin{proof}
(a)$\Rightarrow$(b)
Clearly, $U$ is a direct summand of $S$.

Let $V$ be the left mutation of $S$ at $S/U$, and we show $V=T$.
Since $S$ is the maximal completion of $U$,
we have $H^{-1}(\nu S) \in \calF_S=\calF_U$.
Thus Theorem \ref{Thm_within} implies $V \in \twosilt A$.

Then we have a triangle $S/U \to U' \to V/U \to (S/U)[1]$ with $U' \in \add U$
by Definition-Proposition \ref{Def-Prop_silt_mut},
which induces an exact sequence 
$H^0(S/U) \to H^0(U') \to H^0(V/U) \to 0$.
In particular, $H^0(V/U) \in \Fac H^0(U') \subset \Fac H^0(U)=\calT_U$.
Thus we get $H^0(V) \in \calT_U$; hence $\calT_V \subset \calT_U$.
Since $U \in \add V$, we also have $\calT_U \subset \calT_V$.
Now we obtain $\calT_V=\calT_U$.

Since $V \in \twosilt A$ and $\calT_V=\calT_U$, we get $V=T$.

(b)$\Rightarrow$(a)
Since $S,T \in \twosilt A$ and since
$T$ is the left mutation of $S$ at $S/U$,
Theorem \ref{Thm_within} implies $\ovcalT_S=\ovcalT_U$.
Thus $S$ is the maximal completion of $U$.
Dually, $T$ is the minimal completion of $U$.
\end{proof}

Proposition \ref{Prop_label_filt} can be restated as follows.

\begin{Prop}\cite[Theorem 2.18]{AIR}
Let $U \in \twopsilt A$ with $|U|=|A|-1$,
and take the maximal completion $S$ and the minimal completion $T$.
Then there exists an arrow $S \to T$ in $Q(\twosilt A)$.
All arrows in $Q(\twosilt A)$ arise in this way.
\end{Prop}

By Lemma \ref{Lem_S_T_common} (2) and Theorem \ref{Thm_smc_mut},
we fix the following notations on indecomposable direct summands 
of the completions $S,T \in \twosilt A$ of $U$ 
and the corresponding $X,Y \in \twosmc A$.

\begin{Nota}\label{Nota_S_T}
Let $U=\bigoplus_{i=1}^m U_i \in \twopsilt A$,
and take its maximal and minimal completions $S$ and $T$ respectively.
We fix decompositions 
$S=\bigoplus_{i=1}^n S_i$ and $T=\bigoplus_{i=1}^n T_i$ 
into indecomposable direct summands satisfying the following properties.
\begin{enumerate2}
\item
For each $i \in \{1,\ldots,m\}$, we have $S_i=U_i=T_i$.
\item
For each $j \in \{m+1,\ldots,n\}$, there exists a triangle
$S_j \xrightarrow{f_j} U'_j \to T_j \to S_j[1]$ in $\sfK^\rmb(\proj A)$ with 
$f_j$ a minimal left $(\add U)$-approximation.
\end{enumerate2}
Set $X=\bigoplus_{i=1}^n X_i:=\SH(S)$, $Y=\bigoplus_{i=1}^n Y_i:=\SH(T)$.
Since $J=\{m+1,\ldots,n\}$, we have $\calW_J=\Filt\{X_{m+1},\ldots,X_n\}$.
Thus the following holds.
\begin{enumerate2}
\item[(c)]
For each $i \in \{1,\ldots,m\}$, we have a triangle
$X_i \xrightarrow{g_i} W_i \to Y_i \to X_i[1]$ in $\sfD(A)$ with
$g_i$ a minimal left $\calW_J$-approximation.
\end{enumerate2}
\end{Nota}

By using this notation, we will investigate 
$X=\bigoplus_{i=1}^n X_i$ and $Y=\bigoplus_{i=1}^n Y_i$ in detail.
We start with the following direct consequence of 
Proposition \ref{Prop_X^-_Y^+}.

\begin{Prop}\label{Prop_X^-_Y^+_U}
Let $U=\bigoplus_{i=1}^m U_i \in \twopsilt A$.
\begin{enumerate}
\item
For each $i \in \{1,\ldots,m\}$, 
at least one of $Y_i^+ \ne 0$ and $X_i^- \ne 0$ holds.
\item
For any $j \in \{m+1,\ldots,n\}$, we obtain $X_j^+=Y_j^- \ne 0$
and $X_j^-=Y_j^+=0$.
\end{enumerate}
\end{Prop}

Then we can show that the semibricks $Y^+$ and $X^-$ 
satisfy the following properties.

\begin{Prop}
Let $U=\bigoplus_{i=1}^m U_i \in \twopsilt A$.
\begin{enumerate}
\item
We have the semibricks $Y^+=\bigoplus_{i=1}^m Y_i^+ \in \fLsbrick A$ and 
$X^-=\bigoplus_{i=1}^m X_i^- \in \fRsbrick A$.
\item
We have $\calT_U=\calT_T=\sfT(Y^+)$ and $\calF_U=\calF_S=\sfF(X^-)$.
\end{enumerate}
\end{Prop}

\begin{proof}
(1)
The parts $Y^+=\bigoplus_{i=1}^m Y_i^+$ and $X^-=\bigoplus_{i=1}^m X_i^-$
are Proposition \ref{Prop_X^-_Y^+_U} (2).
Moreover Proposition \ref{Prop_sbrick} (1) implies 
$Y^+ \in \fLsbrick A$ and $X^- \in \fRsbrick A$.

(2)
The parts $\calT_U=\calT_T$ and $\calF_U=\calF_S$ follow from
the definitions of $S$ and $T$.
The remaining parts are Proposition \ref{Prop_sbrick} (2).
\end{proof}

The semibricks $Y^+$ and $X^-$ can be calculated 
without taking the completions $S,T$ of $U$.

\begin{Thm}\label{Thm_endotop_U}
Let $U=\bigoplus_{i=1}^m U_i \in \twopsilt A$.
For each $i \in \{1,\ldots,m\}$, we have
\begin{align*}
Y_i^+=H^0(U_i)/\sum_{f \in \rad_A(H^0(U),H^0(U_i))} \Im f, \quad 
X_i^-=\bigcap_{f \in \rad_A(H^{-1}(\nu U_i),H^{-1}(\nu U))} \Ker f,
\end{align*}
and hence the semibricks $Y^+$ and $X^-$ are
\begin{align*}
Y^+&=H^0(U)/\sum_{f \in \rad_A(H^0(U),H^0(U))} \Im f
=H^0(U)/\rad_{\End_A(H^0(U))}H^0(U), \\
X^-&=\bigcap_{f \in \rad_A(H^{-1}(\nu U),H^{-1}(\nu U))} \Ker f 
=\soc_{\End_A(H^{-1}(\nu U))}H^{-1}(\nu U).
\end{align*}
\end{Thm}

\begin{proof}
By Proposition \ref{Prop_endotop}, 
for each $i \in \{1,\ldots,m\}$, we have
\begin{align*}
Y_i^+=H^0(T_i)/\sum_{f \in \rad_A(H^0(T),H^0(T_i))} \Im f.
\end{align*}
We show 
\begin{align*}
\sum_{f \in \rad_A(H^0(T),H^0(T_i))} \Im f
=\sum_{f \in \rad_A(H^0(U),H^0(U_i))} \Im f.
\end{align*}

The ``$\supset$'' part follows from $U \in \add T$ and $U_i=T_i$.

To show the ``$\subset$'' part, 
it suffices to check that any $f \in \rad_A(H^0(T/U),H^0(T_i))$
admits some $s \in \Z_{\ge 0}$ and
$f' \in \rad_A(H^0(U)^{\oplus s},H^0(U_i))$ such that $\Im f'=\Im f$.
Since $H^0(T/U) \in \calT_{T/U} \subset \calT_T=\calT_U$, 
we can take $s \in \Z_{\ge 0}$ and 
a surjection $g \colon H^0(U)^{\oplus s} \to H^0(T/U)$.
Setting $f':=f \circ g$, we have $\Im f'=\Im f$.
Moreover $g$ is in the radical, since $\add U \cap \add(T/U)=\{0\}$.
Thus we get that $f$ is in the radical.
The ``$\subset$'' part has been proved.

Similarly, by using $\calF_S=\calF_U$, we get 
\begin{align*}
X_i^-=\bigcap_{f \in \rad_A(H^{-1}(\nu S_i),H^{-1}(\nu S))} \Ker f
=\bigcap_{f \in \rad_A(H^{-1}(\nu U_i),H^{-1}(\nu U))} \Ker f.
\end{align*}

These imply the remaining equalities.
\end{proof}

\subsection{$\tau$-Tilting reduction}

For any $j \in \{m+1,\ldots,n\}$, 
Proposition \ref{Prop_X^-_Y^+_U} implies $X_j^+=Y_j^- \ne 0$.
To study these bricks, 
we use \emph{$\tau$-tilting reduction} of Jasso \cite{Jasso}.
This theory focuses on the subcategory and the algebra defined below.

\begin{Def}
Let $U \in \twopsilt A$, 
and $S$ and $T$ be its maximal and minimal completions.
Then we set the \emph{$\tau$-perpendicular subcategory} 
$\calW_U \subset \mod A$ by
\begin{align*}
\calW_U:=\ovcalT_U \cap \ovcalF_U=\ovcalT_S \cap \ovcalF_T
\end{align*}
and a finite dimensional algebra $B$ by
\begin{align*}
B=B_U:=\End_A(H^0(S))/\langle e \rangle,
\end{align*}
where $e \in \End_A(H^0(S))$ is the idempotent $H^0(S) \to H^0(U) \to H^0(S)$.
\end{Def}

The following are some results in $\tau$-tilting reduction,
where $\twosilt_U A$ (resp.~$\twopsilt_U A$) 
is the set consisting of $V \in \twosilt A$ (resp.~$V \in \twopsilt A$)
such that $U \in \add V$.

\begin{Prop}\label{Prop_reduc}(cf.~\cite[Theorem 4.12]{DIRRT})
Let $U \in \twopsilt A$.
\begin{enumerate}
\item
\cite[Proposition 3.6]{Jasso}
The subcategory $\calW_U$ is a wide subcategory of $\mod A$;
that is, $\calW_U$ is an abelian subcategory closed under
extensions in $\mod A$.
\item
\cite[Theorem 3.8]{Jasso}
There exists a category equivalence
\begin{align*}
\Phi \colon \Hom_A(H^0(S),?) \colon \calW_U \to \mod B.
\end{align*}
In particular, $\calW_U$ has exactly $|B|=|A|-|U|$ isoclasses
of simple objects.
\item
\cite[Theorems 3.16, 4.12]{Jasso}
There exist bijections
\begin{align*}
\reduc \colon \twosilt_U A \to \twosilt B, \quad
\reduc \colon \twopsilt_U A \to \twopsilt B 
\end{align*}
such that
\begin{enumerate2}
\item
preserving indecomposable direct summands not in $\add U$;
\item
$\Phi(\calT_V \cap \calW_U)=\calT_{\reduc(V)}$ and
$\Phi(\ovcalT_V \cap \calW_U)=\ovcalT_{\reduc(V)}$
for each $V \in \twopsilt_U A$;
\item
$\reduc(S)=B$, $\reduc(U)=0$, $\reduc(T)=B[1]$.
\end{enumerate2}
\end{enumerate}
\end{Prop}

We remark that $\reduc$ can be extended to a triangle functor
$\sfK^\rmb(\proj A) \to \sfK^\rmb(\proj B)$ \cite[Theorem 4.11]{AHIKM}.
As a direct consequence of Proposition \ref{Prop_reduc} (3), 
$B$ is basic, and decomposed as $B=\bigoplus_{j=m+1}^n \reduc(U \oplus S_j)$
into indecomposable direct summands.
Thus we use the following notation.

\begin{Nota}\label{Nota_P(k)_B}
Let $U \in \twopsilt A$.
For each $k \in \{1,\ldots,n-m\}$, we set
\begin{align*}
P(k)_B:=\reduc(U \oplus S_{m+k}), \quad L(k)_B:=P(k)_B/{\rad P(k)_B}.
\end{align*}
Then $P(1)_B,\ldots,P(n-m)_B$ are 
the nonisomorphic indecomposable projective $B$-modules,
so we adopt $B=\bigoplus_{k=1}^{n-m} P(k)_B$
as the standard decomposition of $B$ as $B$-modules.
Moreover,
$L(1)_B,\ldots,L(n-m)_B$ are the nonisomorphic simple $B$-modules.
\end{Nota}

Recall the exchange quiver $Q(\twosilt A)$ of $\twosilt A$ 
with brick labeling in Definition \ref{Def_label}.
Below we explain that this brick labeling is 
compatible with $\tau$-tilting reduction.
We set $Q(\twosilt_U A)$ as the full subquiver of $Q(\twosilt A)$
whose vertices are the elements of $\twosilt_U A$.

\begin{Prop}\label{Prop_reduc_label}
Let $U \in \twopsilt A$.
\begin{enumerate}
\item
\cite[Corollary 3.17]{Jasso}
There exists an isomorphism $Q(\twosilt_U A) \to Q(\twosilt B)$ 
of quivers induced by $\reduc$.
\item
\cite[Theorem 2.21]{A-semi}
Let $\alpha \colon V \to V'$ be an arrow in $Q(\twosilt_U A)$.
Then its label $L_\alpha$ is in $\calW_U$, 
and the label of the arrow 
$\reduc(\alpha) \colon \reduc(V) \to \reduc(V')$
is $\Phi(L_\alpha)=L_{\reduc(\alpha)}$.
\end{enumerate}
\end{Prop}

Propositions \ref{Prop_reduc} and \ref{Prop_reduc_label} give
the following information.

\begin{Prop}\label{Prop_reduc_source}
Let $U=\bigoplus_{i=1}^m U_i \in \twopsilt A$.
Under Notations \ref{Nota_S_T} and \ref{Nota_P(k)_B}, 
the following statements hold.
\begin{enumerate}
\item
The vertex $S$ is the unique source in $Q(\twosilt_U A)$,
and the arrows starting at $S$ in $Q(\twosilt_U A)$ are
$\alpha_j \colon S \to \mu_j^-(S)$ for $j \in \{m+1,\ldots,n\}$.
\item
For each $j \in \{m+1,\ldots,n\}$, 
we have $\reduc(\mu_j^-(S))=\mu_{j-m}^-(B)$,
and the arrow $\alpha_j \colon S \to \mu_j^-(S)$ in $Q(\twosilt_U A)$ 
is sent to $\reduc(\alpha_j) \colon B \to \mu_{j-m}^-(B)$.
\end{enumerate}
\end{Prop}

\begin{proof}
(1)
Since $Q(\twosilt B)$ has the unique source $B$ 
by Example \ref{Ex_A_source},
$S$ is the unique source of $Q(\twosilt_U A)$ 
by Propositions \ref{Prop_reduc} (3) and \ref{Prop_reduc_label} (1).

Let $\alpha$ be an arrow starting at $S$ in $Q(\twosilt_U A)$.
Since $Q(\twosilt_U A)$ is a full subquiver of $Q(\twosilt A)$,
there exists $j \in \{1,\ldots,n\}$ such that 
the target of $S$ is $\mu_j^-(S)$.
Since $\mu_j^-(S) \in \twosilt_U A$, 
Notation \ref{Nota_S_T} gives $j \in \{m+1,\ldots,n\}$.

Conversely, for any $j \in \{m+1,\ldots,n\}$,
the mutation $V$ of $S$ at $S_j \notin \add U$ is in $\twosilt_U A$.
Thus there exists an arrow between $S$ and $V$ in $Q(\twosilt_U A)$.
The orientation of the arrow is $S \to V$,
since $S$ is a source of $Q(\twosilt_U A)$.

(2)
By definition, $\mu_j^-(S)$ is the left mutation of $S$ at $S_j$,
so $\reduc(\mu_j^-(S))$ is the left mutation of $B$ 
at $\reduc(U \oplus S_j)$
by Propositions \ref{Prop_reduc} (3) and \ref{Prop_reduc_label} (1).
By Notation \ref{Nota_P(k)_B}, we get $\reduc(U \oplus S_j)=P(j-m)_B$. 
These imply $\reduc(\mu_j^-(S))=\mu_{j-m}^-(B)$.
This and (1) give the remaining statement.
\end{proof}

The labels of these arrows give the following properties.
See also \cite[Proposition 4.1]{A}.

\begin{Thm}\label{Thm_sim_W_U}
Let $U \in \twopsilt A$.
\begin{enumerate}
\item
Let $j \in \{m+1,\ldots,n\}$.
Then we have
$X_j^+=Y_j^- \in \simple \calW_U$ and $\Phi(X_j^+)=\Phi(Y_j^-)=L(j-m)_B$.
These bricks label 
the arrows $S \to \mu_j^-(S)$ in $Q(\twosilt_U A)$
and $B \to \mu_{j-m}^-(B)$ in $Q(\twosilt B)$, respectively.
\item
We have $\simple \calW_U=\{X_{m+1}^+,\ldots,X_n^+\}$ and
$\calW_U=\Filt \{X_{m+1}^+,\ldots,X_n^+\}$.
\end{enumerate}
\end{Thm}

\begin{proof}
The equivalence $\Phi \colon \calW_U \to \mod B$ 
in Proposition \ref{Prop_reduc} (2)
sends $\simple \calW_U$ to $\simple (\mod B)$,
which is the set of isomorphic classes of simple $B$-modules.

(1)
Proposition \ref{Prop_X^-_Y^+_U} (2) implies $X_j^+=Y_j^-$. 

By the definition of brick labeling, 
$X_j^+$ labels the arrow $S \to \mu_j^-(S)$ in $Q(\twosilt_U A)$.
Thus Proposition \ref{Prop_reduc_label} (2) 
implies $X_j^+ \in \calW_U$, and
$\Phi(X_j^+)$ is the label of the arrow $\reduc(S) \to \reduc(\mu_j^-(S))$
in $Q(\twosilt B)$.
The latter arrow is $B \to \mu_{j-m}^-(B)$ 
by Proposition \ref{Prop_reduc_source} (2), 
and its label is $L(j-m)_B$ by Example \ref{Ex_A_source}.
Therefore we get $\Phi(X_j^+)=L(j-m)_B$.

Then the first remark and $\Phi(X_j^+)=L(j-m)_B$ imply 
$X_j^+ \in \simple \calW_U$.

Similarly, we obtain $\Phi(Y_j^-)=L(j-m)_B$ and $Y_j^- \in \simple \calW_U$.

(2)
(1) implies that $\Phi$ sends $\{X_{m+1}^+,\ldots,X_n^+\}$ 
to $\{L(1)^B,\ldots,L(n-m)^B\}=\simple (\mod B)$.
Thus the first remark gives 
$\simple \calW_U=\{X_{m+1}^+,\ldots,X_n^+\}$.
We have $\calW_U=\Filt \{X_{m+1}^+,\ldots,X_n^+\}$.
\end{proof}

Then we have the following information on arrows involving $S$ and $T$.

\begin{Prop}\label{Prop_label_S_T}
Let $U \in \twopsilt A$.
\begin{enumerate}
\item
Let $i \in \{1,\ldots,m\}$.
Then there exists at least an arrow $\mu_i^+(S) \to S$ or $T \to \mu_i^-(T)$
in $Q(\twosilt A)$.
Such an arrow cannot be in $Q(\twosilt_U A)$.
\item
In $Q(\twosilt_U A)$, the vertex $S$ is the unique source.
The arrows starting at $S$ in $Q(\twosilt_U A)$ are 
$S \to \mu_j^-(S)$ labeled by the brick $X_j^+=Y_j^-$ 
for $j \in \{m+1,\ldots,n\}$.
\item
In $Q(\twosilt_U A)$, the vertex $T$ is the unique sink.
The arrows ending at $T$ in $Q(\twosilt_U A)$ are 
$\mu_j^+(T) \to T$ labeled by the brick $X_j^+=Y_j^-$ 
for $j \in \{m+1,\ldots,n\}$.
\end{enumerate}
\end{Prop}

\begin{proof}
(1)
By Proposition \ref{Prop_X^-_Y^+_U} (1), 
we have $X_i^- \ne 0$ or $Y_i^+ \ne 0$.
Thus Theorem \ref{Thm_within} implies that
$\mu_i^+(S) \in \twosilt A$ if $X_i^- \ne 0$, 
and $\mu_i^-(T) \in \twosilt A$ if $Y_i^+ \ne 0$.
In the former case, $Q(\twosilt A)$ has an arrow $\mu_i^+(S) \to S$,
and in the latter case, $Q(\twosilt A)$ has an arrow $T \to \mu_i^-(T)$.
Such an arrow cannot be in $Q(\twosilt_U A)$, since $S_i=T_i=U_i$.

(2) follows from Proposition \ref{Prop_reduc_source} 
and Theorem \ref{Thm_sim_W_U}.

(3) is the dual of (2).
\end{proof}

Note that there may be arrows starting at $S$ 
in the whole quiver $Q(\twosilt A)$ other than 
$S \to \mu_j^-(S)$ for $j \in \{m+1,\ldots,n\}$.
Such arrows go outside $Q(\twosilt_U A)$.
The following picture shows the arrows in the quiver $Q(\twosilt A)$ 
around $S$ and $T$, where the gray region denotes $Q(\twosilt_U A)$.
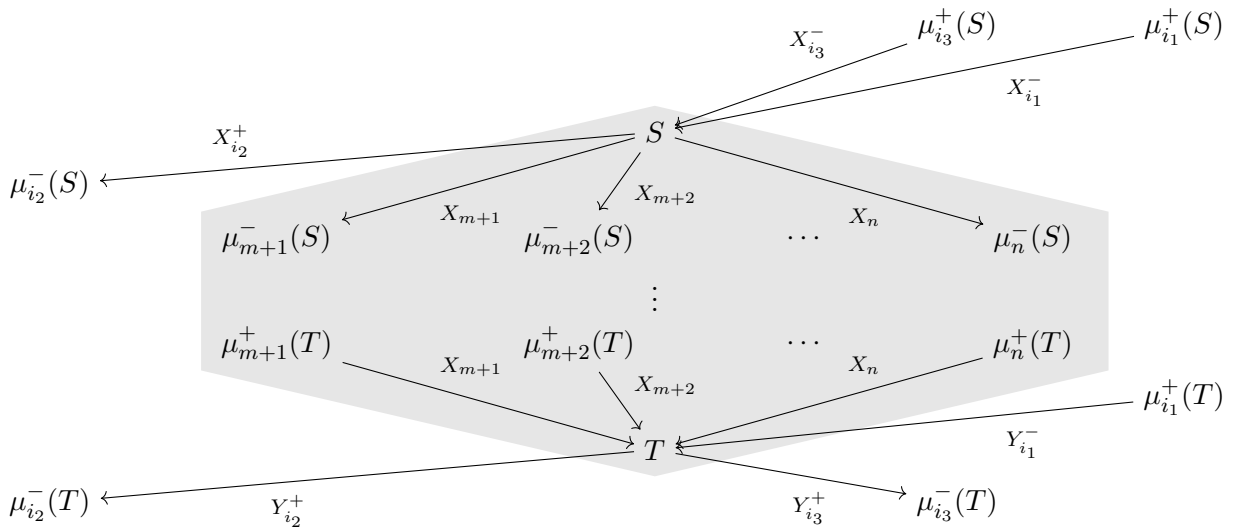
\begin{align*}
\begin{tikzpicture}[->,baseline=0pt,xscale=1, yscale=0.7]
\fill[black!10] (0,3.5)--(-6,1.5)--(-6,-1.5)--
(0,-3.5)--(6,-1.5)--(6,1.5)--cycle;
\node (S0) at ( 0, 3) {$S$};
\node (S1) at ( 7, 5) {$\mu_{i_1}^+(S)$};
\node (S2) at (-8, 2) {$\mu_{i_2}^-(S)$};
\node (S3) at ( 4, 5) {$\mu_{i_3}^+(S)$};
\node (S4) at (-5, 1) {$\mu_{m+1}^-(S)$};
\node (S5) at (-1, 1) {$\mu_{m+2}^-(S)$};
\node (S6) at ( 2, 1) {$\cdots$};
\node (S7) at ( 5, 1) {$\mu_n^-(S)$};
\node (CC) at ( 0, 0) {$\vdots$};
\node (T0) at ( 0,-3) {$T$};
\node (T1) at ( 7,-2) {$\mu_{i_1}^+(T)$};
\node (T2) at (-8,-4) {$\mu_{i_2}^-(T)$};
\node (T3) at ( 4,-4) {$\mu_{i_3}^-(T)$};
\node (T4) at (-5,-1) {$\mu_{m+1}^+(T)$};
\node (T5) at (-1,-1) {$\mu_{m+2}^+(T)$};
\node (T6) at ( 2,-1) {$\cdots$};
\node (T7) at ( 5,-1) {$\mu_n^+(T)$};
\draw (S0) to[edge label ={$\scriptstyle X_{m+1}$},pos=0.7] (S4);
\draw (S0) to[edge label ={$\scriptstyle X_{m+2}$},pos=0.4] (S5);
\draw (S0) to[edge label'={$\scriptstyle X_n$},    pos=0.7] (S7);
\draw (S1) to[edge label ={$\scriptstyle X_{i_1}^-$},pos=0.3] (S0);
\draw (S0) to[edge label'={$\scriptstyle X_{i_2}^+$},pos=0.7] (S2);
\draw (S3) to[edge label'={$\scriptstyle X_{i_3}^-$},pos=0.3] (S0);
\draw (T4) to[edge label ={$\scriptstyle X_{m+1}$},pos=0.3] (T0);
\draw (T5) to[edge label ={$\scriptstyle X_{m+2}$},pos=0.6] (T0);
\draw (T7) to[edge label'={$\scriptstyle X_n$},    pos=0.3] (T0);
\draw (T1) to[edge label ={$\scriptstyle Y_{i_1}^-$},pos=0.3] (T0);
\draw (T0) to[edge label ={$\scriptstyle Y_{i_2}^+$},pos=0.7] (T2);
\draw (T0) to[edge label'={$\scriptstyle Y_{i_3}^+$},pos=0.7] (T3);
\end{tikzpicture}
\end{align*}

From Propositions \ref{Prop_X^-_Y^+_U} (1) and \ref{Prop_label_S_T} (1),
we have information on arrows involving $S$ and $T$ 
not within $Q(\twosilt_U A)$.
For each $i \in \{1,\ldots,m\}$, we have the following three possibilities.
\begin{itemize}
\item[(i)]
There exist arrows $\mu_i^+(S) \to S$ and $\mu_i^+(T) \to T$;
or equivalently, $X_i^- \ne 0$ and $Y_i^- \ne 0$.
\item[(ii)]
There exist arrows $S \to \mu_i^-(S)$ and $T \to \mu_i^-(T)$;
or equivalently, $X_i^+ \ne 0$ and $Y_i^+ \ne 0$.
\item[(iii)]
There exist arrows $\mu_i^+(S) \to S$ and $T \to \mu_i^-(T)$;
or equivalently, $X_i^- \ne 0$ and $Y_i^+ \ne 0$.
\end{itemize}
The indices $i_1,i_2,i_3$ in the picture satisfy 
(i), (ii), (iii), respectively.
Moreover the formulae in Theorem \ref{Thm_endotop_U} allow us 
to calculate the labels of the arrows ending at $S$ or starting at $T$;
namely, $X_{i_1}^-,X_{i_3}^-,Y_{i_2}^+,Y_{i_3}^+$ in the picture.

On the other hand, 
Theorem \ref{Thm_sim_W_U} and Proposition \ref{Prop_label_S_T} (2)
are properties inside $Q(\twosilt_U A)$.
The vertex $S$ is the unique source of $Q(\twosilt_U A)$.
The arrows starting at $S$ in $Q(\twosilt_U A)$ are
labeled by the bricks $X_{m+1},\ldots,X_n$,
which are the simple objects of $\calW_U$.
This corresponds to that there are $n-m$ arrows starting at
the unique source $B$ in $Q(\twosilt B)$,
and that they are labeled by the simple $B$-modules
as in the following picture.
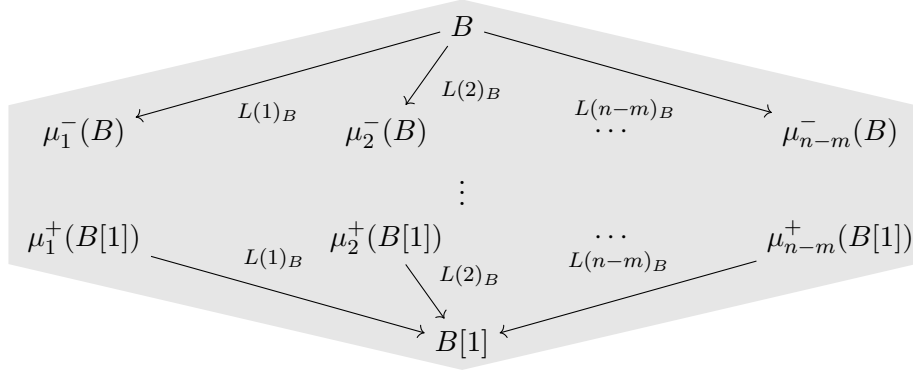
\begin{align*}
\begin{tikzpicture}[->,baseline=0pt,xscale=1, yscale=0.7]
\fill[black!10] (0,3.5)--(-6,1.5)--(-6,-1.5)--
(0,-3.5)--(6,-1.5)--(6,1.5)--cycle;
\node (S0) at ( 0, 3) {$B$};
\node (S4) at (-5, 1) {$\mu_1^-(B)$};
\node (S5) at (-1, 1) {$\mu_2^-(B)$};
\node (S6) at ( 2, 1) {$\cdots$};
\node (S7) at ( 5, 1) {$\mu_{n-m}^-(B)$};
\node (CC) at ( 0, 0) {$\vdots$};
\node (T0) at ( 0,-3) {$B[1]$};
\node (T4) at (-5,-1) {$\mu_1^+(B[1])$};
\node (T5) at (-1,-1) {$\mu_2^+(B[1])$};
\node (T6) at ( 2,-1) {$\cdots$};
\node (T7) at ( 5,-1) {$\mu_{n-m}^+(B[1])$};
\draw (S0) to[edge label ={$\scriptstyle L(1)_B$},  pos=0.7] (S4);
\draw (S0) to[edge label ={$\scriptstyle L(2)_B$},  pos=0.4] (S5);
\draw (S0) to[edge label'={$\scriptstyle L(n-m)_B$},pos=0.7] (S7);
\draw (T4) to[edge label ={$\scriptstyle L(1)_B$},  pos=0.3] (T0);
\draw (T5) to[edge label ={$\scriptstyle L(2)_B$},  pos=0.6] (T0);
\draw (T7) to[edge label'={$\scriptstyle L(n-m)_B$},pos=0.3] (T0);
\end{tikzpicture}
\end{align*}

We end this paper with another description of the algebra $B=B_U$.
Sending each $f \in \End_{\sfD(A)}(S)$ to
$H^0(f) \in \End_{A}(H^0(S))$ gives a homomorphism 
$\phi \colon \End_{\sfD(A)}(S) \to \End_{A}(H^0(S))$
of algebras.
We set $e' \in \End_{\sfD(A)}(S)$ as the idempotent $e' \colon S \to U \to S$,
which satisfies $\phi(e')=e$.
Thus $\phi$ induces a homomorphism 
$\phi \colon \End_{\sfD(A)}(S)/\langle e' \rangle 
\to \End_{A}(H^0(S))/\langle e \rangle$ of algebras,
and the functor $\mod \End_{A}(H^0(S))/\langle e \rangle
\to \mod \End_{\sfD(A)}(S)/\langle e' \rangle$.

\begin{Thm}\label{Psi Phi}
Let $U \in \twopsilt A$.
Then we have $\calW_J=\calW_U$ in $\sfD(A)$, and the commutative diagram
\begin{align*}
\begin{tikzpicture}[->,baseline=(00.base)]
\node (00) at ( 0, 0) {};
\node (-+) at (-3, 1) {$\calW_J$};
\node (++) at ( 3, 1) {$\mod \End_{\sfD(A)}(S)/\langle e' \rangle$};
\node (--) at (-3,-1) {$\calW_U$};
\node (+-) at ( 3,-1) {$\mod \End_A(H^0(S))/\langle e \rangle$};
\draw (--) to[edge label={$\scriptstyle \Hom_A(H^0(S),?)$}] (+-);
\draw (-+) to[edge label={$\scriptstyle \Hom_{\sfD(A)}(S,?)$}] (++);
\draw[double,-] (-+) to (--);
\draw (+-) to[edge label={$\scriptstyle \phi$}] (++);
\end{tikzpicture}
\end{align*}
with all functors being category equivalences.
Therefore $\phi \colon \End_{\sfD(A)}(S)/\langle e' \rangle \to 
\End_A(H^0(S))/\langle e \rangle$ is an isomorphism.
\end{Thm}

\begin{proof}
Set $\calC:=\{X_{m+1}^+,\ldots,X_n^+\}$.
By definition, 
$\calW_J$ is the filtration closure $\Filt_{\sfD(A)} \calC$ of $\calC$ 
in $\sfD(A)$,
and by Theorem \ref{Thm_sim_W_U},
$\calW_U$ is the filtration closure $\Filt_{\mod A} \calC$ of $\calC$ 
in $\mod A$.
Since $\mod A$ is closed under extensions in $\sfD(A)$, 
we have $\calW_J=\calW_U$.

Proposition \ref{Prop_reduc} (2) is nothing but that
the bottom horizontal functor 
$\calW_U \to \mod \End_A(H^0(S))/\langle e \rangle$ is an equivalence.

Setting $E_S:=\End_{\sfD(A)}(S)$, the top horizontal functor 
$\Hom_{\sfD(A)}(S,?) \colon \calW_J \to 
\mod E_S/\langle e' \rangle$
is obtained as a restiction of 
the equivalence $\Hom_{\sfD(A)}(S,?) \colon \calH_S \to \mod E_S$
in Proposition \ref{Prop_E_S},
by the same argument as the proof of \cite[Theorem 3.15]{A-semi}.

From a similar argument to the proof of \cite[Theorem 3.16]{A-semi},
we have the commutativity of the diagram and 
that $\phi$ is isomorphic.
\end{proof}


\begin{thebibliography}{STTVW}
\bibitem[AIR]{AIR} 
T. Adachi, O. Iyama, I. Reiten, 
\emph{$\tau$-tilting theory}, 
Compos. Math. \textbf{150}.3 (2014), 415--452.
\bibitem[Ai]{Ai} 
T. Aihara, 
\emph{Tilting-connected symmetric algebras}, 
Algebr. Represent. Theory \textbf{16}.3 (2013), 873--894.
\bibitem[AI]{AI}
T. Aihara, O. Iyama, 
\emph{Silting mutation in triangulated categories}, 
J. Lond. Math. Soc. (2) \textbf{85}.3 (2012), 633--668.
\bibitem[Al]{Al-Nofayee} 
S. Al-Nofayee, 
\emph{Simple objects in the heart of a $t$-structure}, 
J. Pure Appl. Algebra \textbf{213}.1 (2009), 54--59.
\bibitem[AnS]{AnS}
L. Angeleri H\"{u}gel, F. Sentieri,
\emph{Wide coreflective subcategories and torsion pairs},
J. Algebra \textbf{664} (2025), 164--205.
\bibitem[AHIKM]{AHIKM}
T. Aoki, A. Higashitani, O. Iyama, R. Kase, Y. Mizuno,
\emph{Fans and polytopes in tilting theory I: Foundations},
arXiv:2203.15213v4, to appear in Compos. Math.
\bibitem[As1]{A-semi} 
S. Asai, 
\emph{Semibricks}, 
Int. Math. Res. Not. \textbf{2020}.16 (2020), 4993--5054.
\bibitem[As2]{A} 
S. Asai, 
\emph{The wall-chamber structures of the real Grothendieck groups}, 
Adv. Math. \textbf{381} (2021), Paper No. 107615.
\bibitem[As3]{A-bicpt}
S. Asai,
\emph{Bicompact torsion classes and conjectures on brick infinite algebras},
arXiv:2604.04505v1.
\bibitem[As4]{A-interval}
S. Asai,
\emph{The interval neighborhoods in the real Grothendieck groups},
arXiv:2604.17962v1.
\bibitem[AuS]{AS} 
M. Auslander, S. O. Smal\o, 
\emph{Almost split sequences in subcategories}, 
J. Algebra \textbf{69}.2 (1981), 426--454.
\bibitem[BCZ]{BCZ}
E. Barnard, A. Carroll, S. Zhu,
\emph{Minimal inclusions of torsion classes},
Algebr. Comb. \textbf{2}.5 (2019), 879--901. 
\bibitem[BDH]{BDH}
E. Barnard, C. Defant, E. Hanson,
\emph{Pop-Stack Operators for Torsion Classes and Cambrian Lattices},
arXiv:2312.03959v1.
\bibitem[BBD]{BBD}
A. A. Be\u{\i}linson, J. Bernstein, P. Deligne, 
\emph{Faisceaux pervers},
Analysis and Topology on Singular Spaces, I, Luminy, 1981,
Ast\'{e}risque 100, Soci\'{e}t\'{e} math\'{e}matique de France, 1982, 5--171.
\bibitem[BCPW]{BCPW}
N. Broomhead, R. Coelho Sim\~{o}es, D. Pauksztello, J. Woolf,
\emph{Simple tilts of length hearts and simple-minded mutation},
arXiv:2401.02947v2.
\bibitem[BY]{BY} 
T. Br\"ustle, D. Yang, \emph{Ordered exchange graphs}, 
Advances in representation theory of algebras, 135--193, 
EMS Ser. Congr. Rep., Eur. Math. Soc., Z\"urich, 2013.
\bibitem[DW]{DW}
C. Defant, N. Williams, 
\emph{Semidistrim lattices}, 
Forum Math. Sigma \textbf{11}.50 (2023), 1--35.
\bibitem[D]{Demonet}
L. Demonet,
\emph{Combinatorics of Mutations in Representation Theory}, 
l'Habilitation \`{a} Diriger des Recherches, 
Laboratoire de Math\'{e}atiques Nicolas Oresme, 
Universit\'{e} de Caen, 2017, 
\url{https://www.math.nagoya-u.ac.jp/~demonet/recherche/habilitation.pdf}.
\bibitem[DIJ]{DIJ}
L. Demonet, O. Iyama, G. Jasso, 
\emph{$\tau$-Tilting Finite Algebras, Bricks, and g-Vectors}, 
Int. Math. Res. Not. \textbf{2019}.3 (2019), 852--892.
\bibitem[DIRRT]{DIRRT} 
L. Demonet, O. Iyama, N. Reading, I. Reiten, H. Thomas, 
\emph{Lattice theory of torsion classes: Beyond $\tau$-tilting theory},
Trans. Amer. Math. Soc. Ser. B \textbf{10} (2023), 542--612.
\bibitem[DF]{DF} 
H. Derksen, J. Fei, 
\emph{General presentations of algebras}, 
Adv. Math. \textbf{278} (2015), 210--237.
\bibitem[E]{Enomoto}
H. Enomoto,
\emph{Monobrick, a uniform approach to torsion-free classes 
and wide subcategories},
Adv. Math. \textbf{393} (2021), 108113.
\bibitem[FZ]{FZ} 
S. Fomin, A. Zelevinsky, 
\emph{Cluster algebras. I. Foundations}, 
J. Amer. Math. Soc. \textbf{15}.2 (2002), 497--529.
\bibitem[F]{Fushimi}
R. Fushimi,
\emph{The correspondence between silting objects and 
$t$-structures for non-positive dg algebras},
arXiv:2312.17597v3.
\bibitem[GZ]{GZ}
E. Gupta, Y. Zhou, 
\emph{Semibricks and wide subcategories in extended module categories},
arXiv:2511.08157v1.
\bibitem[HW]{HW}
W. Hara, M. Wemyss,
\emph{Spherical objects in dimensions two and three},
J. Eur. Math. Soc. \textbf{28}.4 (2026), 1771--1808.
\bibitem[IY]{IY}
O. Iyama, D. Yang, 
\emph{Silting reduction and Calabi–Yau reduction of triangulated categories}, 
Trans. Amer. Math. Soc. \textbf{370} (2018), 7861--7898.
\bibitem[Ja]{Jasso} 
G. Jasso, 
\emph{Reduction of $\tau$-tilting modules and torsion pairs},
Int. Math. Res. Not. \textbf{2015}.16 (2015), 7190--7237. 
\bibitem[Ji]{Jin}
H. Jin,
\emph{Reductions of triangulated categories and simple-minded collections},
J. London Math. Soc. (2) \textbf{107}.4 (2023), 1482--1518.
\bibitem[J\o]{Jorgensen}
P. J{\o}rgensen, 
\emph{Torsion classes and t-structures in higher homological algebra},
Int. Math. Res. Not. \textbf{2016}.13 (2016), 3880--3905. 
\bibitem[Ka]{Kalck}
M. Kalck,
\emph{A finite dimensional algebra with a phantom
(a corollary of an example by J. Krah)},
arXiv:2304.08417v1.
\bibitem[KN]{KN}
B. Keller and P. Nicol\'{a}s, 
\emph{Weight structures and simple dg modules for positive dg algebras}, 
Int. Math. Res. Not. \textbf{2013}.5 (2013), 1028--1078.
\bibitem[KV]{KV} 
B. Keller, D. Vossieck, 
\emph{Aisles in derived categories}, 
Deuxi\`{e}me Contact Franco-Belge en Alg\`{e}bre (Faulx-les-Tombes, 1987), 
Bull. Soc. Math. Belg. S\'{e}r. A \textbf{40}.2 (1988), 239--253.
\bibitem[KY]{KY} 
S. Koenig, D. Yang,
\emph{Silting objects, simple-minded collections, 
t-structures and co-t-structures for finite dimensional algebras},
Doc. Math. \textbf{19} (2014), 403--438.
\bibitem[Kr]{Krah}
J. Krah,
\emph{A phantom on a rational surface}, 
Invent. math. \textbf{235} (2024), 1009--1018. 
\bibitem[MP]{MP}
K. Mousavand, C. Paquette,
\emph{Hom-orthogonal modules and brick-Brauer-Thrall conjectures},
J. Algebra \textbf{686} (2026), 650--676.
\bibitem[Pf]{Pfeifer} 
C. Pfeifer,
\emph{Remarks on $\tau$-tilted versions 
of the second Brauer-Thrall Conjecture},
Bull. Lond. Math. Soc. \textbf{57}.5 (2025), 1568--1583.
\bibitem[Pl]{Plogmann}
M. Plogmann,
\emph{Complicial simple-minded collections},
arXiv:2603.03122v2.
\bibitem[R1]{Rickard}
J. Rickard, 
\emph{Morita theory for derived categories},
J. London Math. Soc. (2) \textbf{39}.3 (1989), 436--456.
\bibitem[R2]{Rickard-symm}
J. Rickard,
\emph{Equivalences of derived categories for symmetric algebras},
J. Algebra \textbf{257} (2002), 460--481.
\bibitem[STV]{STV}
S. Schroll, H. Treffinger, Y. Valdivieso,
\emph{On band modules and $\tau$-tilting finiteness},
Math. Z. \textbf{299} (2021), 2405--2417.
\bibitem[S]{Sentieri}
F. Sentieri,
\emph{A brick version of a theorem of Auslander},
Nagoya. Math. \textbf{249} (2023), 88--106.
\bibitem[SZ]{SZ}
Y. Sun, Y. Zhang,
\emph{Gluing simple-minded collections in triangulated categories},
J. Algebra \textbf{645} (2024), 54--85.
\bibitem[Z]{Zhang}
Y. Zhang,
\emph{Reduction of Wide Subcategories and Recollements},
Algebra Colloq. \textbf{30}.4 (2023), 713--720.
\end{thebibliography}
\end{document}